\documentclass[11pt]{amsart}
\usepackage{hyperref}
\usepackage{amssymb,amsmath,amsfonts,latexsym}
\usepackage{bm}
\usepackage{mathtools}
\allowdisplaybreaks
\usepackage{enumerate}

\usepackage{array,graphics,color}
\numberwithin{equation}{section}
\newtheorem{dfn}{Definition}[section]
\newtheorem{thm}[dfn]{Theorem}

\newtheorem{ppsn}[dfn]{Proposition}
\newtheorem{crlre}[dfn]{Corollary}
\newtheorem{xmpl}[dfn]{Example}
\newtheorem{rmrk}[dfn]{Remark}

\newcommand{\D}{\mathbb{D}}
\newcommand{\C}{\mathbb{C}}		
\newcommand{\fcl}{\mathcal{F}}
\newcommand{\mcl}{\mathcal{M}}
\newcommand{\wcl}{\mathcal{W}}

\newcommand{\kcl}{\mathcal{K}}

\newcommand{\hdcr}{{H}^2_{\mathbb{C}^r}(\mathbb{D})}
\newcommand{\hdcmone}{H^2_{\mathbb{C}^{m+1}}(\mathbb{D})}
\newcommand{\hdcm}{{H}^2_{\mathbb{C}^m}(\mathbb{D})}
\newcommand{\hdcc}{{H}^2_{\mathbb{C}}(\mathbb{D})}

\DeclarePairedDelimiterX{\norm}[1]{\lVert}{\rVert}{#1}

\begin{document}
	
	\title[Kernels of Perturbed Toeplitz Operators in vector-valued Hardy spaces]{Kernels of Perturbed Toeplitz Operators in vector-valued Hardy spaces}
	
	
	\author[Chattopadhyay] {Arup Chattopadhyay}
	\address{Department of Mathematics, Indian Institute of Technology Guwahati, Guwahati, 781039, India}
	\email{arupchatt@iitg.ac.in, 2003arupchattopadhyay@gmail.com}

	\author[Das]{Soma Das}
	\address{Department of Mathematics, Indian Institute of Technology Guwahati, Guwahati, 781039, India}
	\email{soma18@iitg.ac.in, dsoma994@gmail.com}
	
	\author[Pradhan]{Chandan Pradhan}
	\address{Department of Mathematics, Indian Institute of Technology Guwahati, Guwahati, 781039, India}
	\email{chandan.math@iitg.ac.in, chandan.pradhan2108@gmail.com}

	
	
	\begin{abstract}
	Recently, Liang and Partington \cite{YP} show that kernels of finite-rank perturbations of Toeplitz operators are nearly invariant with finite defect under the backward shift operator acting on the scalar-valued Hardy space. In this article we provide a vectorial generalization of a result of Liang and Partington. As an immediate application we identify the kernel of perturbed Toeplitz operator in terms of backward shift-invariant subspaces in various important cases by applying the recent theorem (\cite{CDP, OR}) in connection with nearly invariant subspaces of finite defect for the backward shift operator acting on the vector-valued Hardy space .
    \end{abstract}
	\maketitle
	\section{Introduction}
	 It is well known that the kernel of a Toeplitz operator is nearly invariant under the backward shift operator acting on the scalar-valued Hardy space and the concept of nearly backward shift invariant subspaces was first introduced by Hitt in \cite{HIT} as a generalization to Hayashi’s results concerning Toeplitz kernels in \cite{HE}. Later Sarason \cite{SR} further investigated these spaces and modified Hitt’s algorithm for scalar-valued Hardy space to study the kernels of Toeplitz operators. In 2010, Chalendar-Chevrot-Partington (C-C-P) \cite{CCP} gives a complete characterization of nearly invariant subspaces under the backward shift operator acting on the vector-valued Hardy space, providing a vectorial generalization of a result of Hitt. Recently, Chalendar-Gallardo-Partington (C-G-P) \cite{CGP} introduce the notion of nearly invariant subspace of finite defect for the backward shift operator acting on the scalar valued Hardy space as a generalization of nearly invariant subspaces and provides a complete characterization of these spaces in terms of backward shift invariant subspaces.  A recent preprint \cite{CDP} by the authors of this article characterizes nearly invariant subspace of finite defect for the backward shift operator acting on the vector-valued Hardy space, providing a vectorial generalization of a result of Chalendar-Gallardo-Partington (C-G-P). In this connection it is worth mentioning that similar kind of connection independently obtained by Ryan O\textquoteright Loughlin in \cite{OR}. Furthermore in this context, Liang and Partington \cite{YP} recently provide a connection between kernels of finite-rank perturbations of Toeplitz operators and nearly invariant subspaces with finite defect under the backward shift operator acting on the scalar-valued Hardy space. In other words they give an affirmative answer of the following question which is closely related with the invariant subspace problem:
    \begin{equation} \label{eq1}
     \begin{split}
     & \textup{\emph{Given a Toeplitz operator T acting on the scalar-valued Hardy space, does the kernel of a}}\\
     & \textup{\emph{finite-rank perturbation of T is nearly backward shift invariant with finite defect ? }}
     \end{split}
\end{equation}
Moreover, they also identify the kernel of perturbed Toeplitz operator in terms of backward shift-invariant subspaces in several important cases by applying a recent theorem by Chalendar--Gallardo--Partington (C-G-P).

The purpose of this paper is to study the kernels of finite-rank perturbations of Toeplitz operators and its connection with nearly invariant subspaces with finite defect under the backward shift operator acting on the vector-valued Hardy space. In other words we give an affirmative answer of the above question \eqref{eq1} in the vector-valued Hardy space, providing a vectorial generalization of a result of Liang and Partington. Furthermore, we also identify the kernel of perturbed Toeplitz operator in terms of backward shift-invariant subspaces by applying our recent theorem (see \cite[Theorem 3.5]{CDP}) in connection with nearly invariant subspaces of finite defect for the backward shift operator acting on the vector-valued Hardy space in several important cases as mentioned by Liang and Partington in \cite{YP}. For more information on this direction of research, we refer the reader to (\cite{HE,NC, T}) and the references therein. In order to state the precise contribution of this paper, we need to introduce first some definitions and notations.   

The $\mathbb{C}^m$- valued Hardy space \cite{JP} over the unit disc $\mathbb{D}$ is denoted by $\hdcm$ and defined by $$\hdcm:=\Big\{F(z)=\sum_{n\geq 0} A_nz^n:~\|F\|^2= \sum_{n\geq 0}~\norm{A_n}_{\mathbb{C}^m}^2<\infty,~A_n\in\mathbb{C}^m\Big\}.$$ We can also view the above Hilbert space as the direct sum of $m$-copies of $H^2_{\mathbb{C}}(\D)$ or  sometimes it is useful to see the above space as a tensor product of two Hilbert spaces $H^2_{\mathbb{C}}(\D)$ and $\mathbb{C}^m$, that is, $$H^2_{\mathbb{C}^m}(\D)  \equiv   \underbrace{H^2_{\mathbb{C}}(\D)\oplus \cdots\oplus H^2_{\mathbb{C}}(\D)}_{m}\equiv H^2_{\mathbb{C}}(\D)\otimes \mathbb{C}^m.$$
On the other hand the space $\hdcm$ can also be defined as the collection of all $\mathbb{C}^m$-valued analytic functions $F$ on $\mathbb{D}$ such that 
$$\norm{F}= \Big[~\sup_{0\leq r<1} \frac{1}{2\pi} \int_0^{2\pi} \vert F(re^{i\theta})\vert^2~ d\theta~\Big]^{\frac{1}{2}}<\infty.$$ 
Moreover the nontangential boundary limit (or radial limit) $$F(e^{i\theta}):= \lim\limits_{r\rightarrow 1-}F(re^{i\theta})$$ exists almost everywhere on the unit circle $\mathbb{T}$ (for more details see \cite{NB}, I.3.11). Therefore $\hdcm$ can be embedded isomertically as a closed subspace of
$L^2(\mathbb{T},\mathbb{C}^m)$ by identifying $\hdcm$ through the nontangential boundary limits of the $\hdcm$ functions. Furthermore $L^2(\mathbb{T},\mathbb{C}^m)$ can be decomposed in the following way 
$$L^2(\mathbb{T},\mathbb{C}^m)=\hdcm \oplus \overline{ H _0^2},$$ where $\overline{ H _0^2}=\{F\in L^2(\mathbb{T},\mathbb{C}^m):\overline{F}\in \hdcm ~and~ F(0)=0 \}$. 
In other words in the above decomposition $\hdcm$ is identified with the subspace spanned by $\{e^{int}:n\geq 0\}$  and  $\overline{ H _0^2}$ is the subspace spanned by $\{e^{int}:n<0\}$, respectively. 
Let $S$ denote the forward shift operator  (multiplication by the independent variable)
acting on $\hdcm$, that is, $SF(z)=zF(z)$,~$z\in \D$. The adjoint of $S$ is denoted by $S^*$ and defined in $\hdcm$ as the operator 
$$S^*(F)(z)=\dfrac{F(z)-F(0)}{z},~~ F\in\hdcm$$ which is known as backward shift operator.
The Banach space of all $\mathcal{L}(\mathbb{C}^r,\mathbb{C}^m)$ (set of all bounded linear operators from $\mathbb{C}^r$ to $\mathbb{C}^m$)- valued  bounded analytic functions on $\mathbb{D}$ is denoted by 
$H^{\infty}_{\mathcal{L}(\mathbb{C}^r,\mathbb{C}^m)}(\D)$ and the associated norm is $$\norm F_{\infty} = \sup_{z\in \D} ~\norm {F(z)}.$$ Moreover, the space 
$H^{\infty}(\D,\mathcal{L}(\mathbb{C}^r,\mathbb{C}^m))$ can be embedded isometrically as a closed subspace of $L^\infty(\mathbb{T},\mathcal{L}(\mathbb{C}^r,\mathbb{C}^m))$.
Note that each $\Theta\in H^{\infty}_{\mathcal{L}(\mathbb{C}^r,\mathbb{C}^m)}(\D)$ induces a bounded linear map $T_{\Theta}\in H^{\infty}_{\mathcal{L}(\mathbb{C}^r,\mathbb{C}^m)}(\D)$ defined by
$$T_{\Theta}F(z)=\Theta(z)F(z).~~(F\in H^{2}_{\mathbb{C}^r}(\D))$$
The elements of $H^{\infty}_{\mathcal{L}(\mathbb{C}^r,\mathbb{C}^m)}(\D)$  are called the \emph{multipliers} and are determined by
$$\Theta\in H^{\infty}_{\mathcal{L}(\mathbb{C}^r,\mathbb{C}^m)}(\D)\textit{~if~and~only~if~} ST_{\Theta}=T_{\Theta}S,$$
where the shift $S$ on the left hand side and the right hand side act on $ H ^{2}_{\mathbb{C}^m}(\D)$ and $ H ^{2}_{\mathbb{C}^r}(\D)$ respectively. A multiplier $\Theta \in  H ^{\infty}_{\mathcal{L}(\mathbb{C}^r,\mathbb{C}^m)}(\D)$ is said to be \emph{inner} if $T_{\Theta}$ is an isometry, or equivalently, 
$\Theta(e^{it})\in \mathcal{L}(\mathbb{C}^r,\mathbb{C}^m)$ is an isometry almost everywhere with respect to the Lebesgue measure on $\mathbb{T}$. Inner multipliers are among the most important tools for classifying invariant subspaces of reproducing kernel Hilbert spaces. For instance: 

\begin{thm}\label{a1}
	(Beurling-Lax-Halmos \cite{NF}) 
	A non-zero closed subspace $\mathcal{M}\subseteq  H ^{2}_{\mathbb{C}^m}(\D)$ is shift
	invariant if and only if there exists an inner multiplier $\Theta \in H ^{\infty}_{\mathcal{L}(\mathbb{C}^r,\mathbb{C}^m)}(\D)$ such that
	$$\mathcal{M} = \Theta  H ^{2}_{\mathbb{C}^r}(\D),$$
	for some $r$ ($1\leq r\leq m$).  
\end{thm}
Consequently, the space $\mathcal{M}^{\perp}$ of 
$H^2_{\mathbb{C}^m}(\D)$ which is invariant under $S^*$ (backward shift) can be represented as 
$$\mathcal{K}_{\Theta}: = \mathcal{M}^{\perp} =  H^2_{\mathbb{C}^m}(\D) \ominus \Theta  H^2_{\mathbb{C}^r}(\D),$$
which also known as model spaces (\cite{EMV1,EMV2,NB,RSN}).
Let $P_m:L^2(\mathbb{T},\mathbb{C}^m) \to H^2_{C^m}(\D)$ be an orthogonal projection onto $H^2_{C^m}(\D)$ defined by $$ \sum_{n=-\infty}^{\infty}A_ne^{int}\mapsto\sum_{n=0}^{\infty}A_ne^{int}.$$ Therefore $P_m(F)=(Pf_1,Pf_2,\ldots ,Pf_m),$ where $P$ is the Riesz projection on $\hdcc$ \cite{EMV1} and $F=(f_1,f_2,\ldots,f_m)\in L^2(\mathbb{T},\mathbb{C}^m)$.
Also note that for any $\Phi \in L^\infty(\mathbb{T},\mathcal{L}(\mathbb{C}^m,\mathbb{C}^m))$, the Toeplitz operator $T_{\Phi}:\hdcm \to \hdcm$ is defined by $$T_{\Phi}(F)=P_m(\Phi F)$$ for any $F\in \hdcm$. Since $\hdcm$ can be written as direct sum of m-copies of $\hdcc$, then we have the following matrix-representation of $T_{\Phi}:$
\begin{equation}
T_{\Phi}=
T_{\Phi}=
\begin{bmatrix}  
T_{\phi _{11}}&T_{\phi _{12}}&\cdots&T_{\phi _{1m}} \\[1pt]
T_{\phi _{21}}& T_{\phi _{22}}&\cdots&T_{\phi _{2m}} \\
\vdots&\vdots&\ddots&\vdots\\
T_{\phi _{m1}}& T_{\phi _{m2}}&\cdots& T_{\phi _{mm}} \\
\end{bmatrix}_{m\times m}, 
\end{equation}
where
\begin{equation}
\Phi=
\begin{bmatrix}  
\phi _{11}&\phi _{12}&\cdots&\phi _{1m} \\[1pt]
\phi _{21}&\phi _{22} &\cdots&\phi _{2m} \\
\vdots&\vdots&\ddots&\vdots\\
\phi _{m1}&\phi _{m2}&\cdots&\phi _{mm} \\
\end{bmatrix}_{m\times m}
\end{equation}
is an element of $L^\infty(\mathbb{T},\mathcal{L}(\mathbb{C}^m,\mathbb{C}^m))$ and each $T_{\phi _{ij}}$ is a Toeplitz operator in $\hdcc$ \cite{RIW}. It is well known that $T_\Phi ^*=T_{\Phi ^*}$. Next we introduce the definition of nearly invariant subspaces for $S^*$ in vector valued Hardy space.

\begin{dfn}
A closed subspace $\mathcal{M}$ of $\hdcm$ is said to be nearly invariant for $S^*$ if every element  $F\in \mathcal{M}$ with $F(0)=0$ satisfies $S^*F\in \mathcal{M}$. Moreover a closed subspace $\mcl \subset \hdcm$ is said to be nearly $S^*$-invariant with defect p if and only if there is an $p$-dimensional subspace $\fcl \subset \hdcm$ (which may be taken to be orthogonal to $\mcl$ ) such that if $F \in M,F(0)=0$ then $S^*F$ $\in M\oplus \fcl$, this subspace $\fcl$ is called the defect space.	
\end {dfn}
In our earlier work \cite {CDP} we have shown a connection between nearly $S^*$ invariant subspaces and $S^*$ invariant subspaces in vector valued Hardy space.  It can be easily seen that the kernel of a Toeplitz operator is nearly $S^*$ invariant in vector valued Hardy space. Next we consider a finite rank perturbation (say rank n) of a Toeplitz operator $T_{\Phi}:\hdcm \to \hdcm$, denoted by $T_n$ and defined by as follows: $$T_n(F)=T_{\Phi}(F)+ \sum_{i=1}^{n}\langle F,G_i\rangle H_i~,~~~\forall F\in \hdcm,$$ where $\{G_i\}_{i=1}^{n}$ and $\{H_i\}_{i=1}^n$ are orthonormal sets in $\hdcm$. Therefore it is  natural to ask whether the kernel of $T_n$ is nearly $S^*$ invariant subspace  with finite defect or not. In this article we provide an affirmative answer of this question in several important cases mentioned by Liang and Partington in \cite{YP}. In other words we solve the above mentioned problem \eqref{eq1} in various important cases in vector-valued Hardy spaces, providing a vectorial generalization of a result of Liang and Partington \cite{YP}. For simplicity we first discuss the problem for rank two perturbation, that is for $T_2$ and then we  state our main theorem for rank n perturbation of $T_{\Phi}$, that is for $T_n$.

The rest of the paper is organized as follows: In Section 2, we study the kernel of $T_n$ whenever $\Phi =0$ almost everywhere on the circle and some applications of our earlier theorem \cite[ Theorem 3.5]{CDP}. Section 3,4 and 5 deal with the study of the kernel of $T_n$ in several important cases as mentioned by Liang and Partington \cite{YP} whenever $\Phi$ is non zero almost everywhere on the circle and few applications of our earlier theorem \cite[ Theorem 3.5]{CDP}.

\section{Kernel of finte rank perturbation of Toeplitz operator having symbol zero almost everywhere on the circle}
In this section, we will study the kernel of $T_n=T_{\Phi}+\sum_{i=1}^{n}\langle .,G_i\rangle H_i~~,$ where $\Phi =0$ almost everywhere on $\mathbb{T}$. As we have discussed earlier first we study the kernel of $T_2$ and later we will state the main theorem corresponding to $T_n$. Note that if $\Phi=0$ almost everywhere on $\mathbb{T}$, then the kernel of $T_2$ is given by $$Ker T_2=\hdcm \ominus \bigvee \{G_1,G_2\}.$$ It is easy to check that the kernel of $T_2$ is nearly $S^*$ invariant with defect 2 because if we consider any element $F\in Ker T_2$ with $F(0)=0$, then $S^*(F)\in Ker T_2 \cup (\bigvee\{G_1,G_2\})=\hdcm $ and the defect space is $\fcl =\bigvee\{G_1,G_2\}.$ Now for the general case we have the following result:
\begin{thm}\label{b}
Suppose $\Phi =0$ almost everywhere on $\mathbb{T}$. Then the subspace $Ker T_n$ is nearly $S^*$ invariant with defect $n$ and the defect space is $\fcl =\bigvee \{G_1,G_2,...,G_n\}$. 
\end{thm}  
Next we will see a nice application of the following theorem obtained by us \cite{CDP} as well as independently by Ryan O’Loughlin in \cite{OR} to understand the kernel of perturbed Toeplitz operator in a better way in terms of backward shift invariant subspaces. Before that let us recall that theorem concerning nearly $S^*$ invariant subspaces with defect $p$ on vector valued Hardy spaces $\hdcm$.

\begin{thm} \cite [ Theorem 3.5 ]{CDP} (see also \cite[Theorem 3.4]{OR})\label{a}
	Let $\mcl$ be a closed subspace that is nearly $S^*$-invariant with defect $p$ in $\hdcm$ and let $\{E_1,E_2,.....E_p \}$ be any orthonormal basis for the $p$-dimensional defect space $\fcl$. Let $\{W_1,W_2,\ldots,W_r\}$ be an orthonormal basis of $\wcl:=\mcl\ominus(\mcl\cap z\hdcm)$ and let $F_0$ be the $m\times r$ matrix whose columns are $W_1,W_2,\ldots,W_r$. Then\vspace*{0.1in}\\
	(i) in the case where there are functions in $\mcl$ that do not vanish at $0$,
	\begin{equation}\label{main1}
	\mcl= \Big\{F : F(z)= F_0(z)K_0(z)+ \sum_{j=1}^{p} zk_j(z)E_j(z) : (K_0,k_1,\ldots,k_p)\in \kcl \Big\},
	\end{equation}
	where 
	$\kcl \subset \hdcr \times \underbrace{\hdcc\times\cdots\times \hdcc}_{p} $ is a closed $S^*\oplus\cdots\oplus S^*$- invariant subspace of the vector valued Hardy space $ H ^2_{\mathbb{C}^{r+p}}(\mathbb{D})$ and $$\norm{F}^2=\norm{K_0}^2+\sum_{j=1}^{p}\norm{k_j}^2 .$$
	
	\noindent (ii) In the case where all the functions in $\mcl $ vanish at $0$,
	\begin{equation}\label{main2}
	\mcl= \Big\{F : F(z)= \sum_{j=1}^{p} zk_j(z)E_j(z) : (k_1,\ldots,k_p)\in \mathcal{K} \Big\},
	\end{equation}
	with the same notion as in $(i)$ except that $\mathcal{K}$ is now a closed 	$S^*\oplus \cdots\oplus S^*$- invariant subspace of the vector valued Hardy space $ H ^2_{\mathbb{C}^{p}}(\mathbb{D})$ and $$\norm{F}^2=\sum_{j=1}^{p}\norm{k_j}^2 .$$
	Conversely, if a closed subspace $\mcl$ of the vector valued Hardy space $\hdcm$ has a representation like $(i)$ or $(ii)$ as above, then it is a nearly $S^*$-invariant subspace of defect $p$.
\end{thm}
 For simplicity we will deal with rank one perturbation of Toeplitz operator and let it be denoted by $T$ , that is $$T=T_{\Phi}+\langle .,G\rangle H$$ with $\norm G_2=1~and~S^*H\neq 0$. Now as a application of the above Theorem our aim is to represent the kernel of the operator $T$ in some special cases. It should also be observed that we can find $\kcl$ as the largest $S^*$ invariant subspace in $\hdcm $ like the scalar case such as $$F_0(z)S^{*n}K_0(z)+z \sum_{j=1}^{p} S^{*n}k_j(z)E_j(z) \in \mcl ~or~z\sum_{j=1}^{p} S^{*n}k_j(z)E_j(z) \in \mcl$$ for all $n\in \mathbb{N} \cup \{0\}$.
 In the case $\Phi =0$, we have $$\mcl =Ker T=\hdcm \ominus \langle G\rangle$$ 
which is  nearly $S^*$ invariant with defect space $\fcl =\langle G\rangle$ by Theorem[\ref{b}]. Assume furthermore that the function $G\in H^\infty(\D,\mathcal{L}(\C,\C^m))=H_{\C^m}^\infty(\D)$.
Now consider $F_i=P_{\mcl}(k_0 \otimes e_i),$ where $k_0$ is the reproducing kernel at $0$ and $\{e_i:1\leq i\leq m\}$ is a standard orthonormal basis of $\C^m,$ generate the subspace $\wcl$ in Theorem [\ref{a}]. Without loss of generality we assume that $\{F_1,F_2,\ldots F_r\}~(where~r\leq m$) is a basis of $\wcl$.
By using Gram-Schmidt orthonormalization we find an orthonormal basis of $\wcl$ as follows:
$W_1=C_{11}F_1,~W_2=C_{21}F_1+C_{22}F_2,~\ldots ,~W_r=C_{r1}F_1+C_{r2}F_2+\ldots +C_{rr}F_r$, where the constant $C_{ij}$ can be determined via the process of orthonormalization. Now if we consider  $G=(g_1,g_2,\ldots ,g_m)\in H_{\C^m}^\infty(\D)~with~\norm G _2=1$ and $\mcl=\hdcm \ominus \langle G\rangle$, then for any $i\in \{1,2,\ldots ,m\}$ we have $$F_i=P_{\mcl}(k_0\otimes e_i)=(-\overline{g_i(0)}g_1,-\overline{g_i(0)}g_2,\ldots ,1-\overline{g_i(0)}g_i,\ldots ,-\overline{g_i(0)}g_m).$$ Therefore by Theorem\ref{a}, the $m\times r$ matrix $F_0$ whose columns are $\{W_1,W_2,\ldots ,W_r\}$ has the following representation
\begin{equation}
F_0=
\begin{bsmallmatrix}  
C_{11}(1-\overline{g_1(0)}g_1)& C_{21}(1-\overline{g_1(0)}g_1)+C_{22}(-\overline{g_2(0)}g_1)&\cdots&C_{r1}(1-\overline{g_1(0)}g_1)+\cdots +C_{rr}(-\overline{g_r(0)}g_1)\\
C_{11}(-\overline{g_1(0)}g_2)& C_{21}(-\overline{g_1(0)}g_2)+C_{22}(1-\overline{g_2(0)}g_2)&\cdots& C_{r1}(-\overline{g_1(0)}g_2)+\cdots +C_{rr}(-\overline{g_r(0)}g_2) \\
\vdots&\vdots&\ddots&\vdots\\
C_{11}(-\overline{g_1(0)}g_m) & C_{21}(-\overline{g_1(0)}g_m)+C_{22}(-\overline{g_2(0)}g_m)&\cdots&C_{r1}(-\overline{g_1(0)}g_m)+\cdots +C_{rr}(-\overline{g_r(0)}g_m)\\
\end{bsmallmatrix}_{m\times r}
\end{equation}
According to Theorem\ref{a}, $\mcl$ has the following representation :
\begin{enumerate}[(1)]
	\item In the case when $\wcl \neq \{0\}$, $$\mcl =\{F:F(z)=F_0(z)K_0(z)+zk_1(z)G(z):(K_0,k_1)\in \kcl \subseteq \hdcr \times \hdcc\}.$$ Suppose $ |G|^2=|g_1|^2+|g_2|^2+\ldots +|g_m|^2$ and if we consider
	\begin{align}
	G_0 &= \begin{bmatrix}
	\overline{C}_{11}P(g_1-g_1(0)|G|^2) \\
	\overline{C}_{21}P(g_1-g_1(0)|G|^2)+\overline{C}_{22}P(g_2-g_2(0)|G|^2) \\
	\vdots \\
	\overline{C}_{r1}P(g_1-g_1(0)|G|^2)+\cdots +\overline{C}_{rr}P(g_r-g_r(0)|G|^2)
	\end{bmatrix} \in\hdcr
	\end{align}
	and $g=P(\overline{z}|G|^2) \in \hdcc $, then the $S^*\oplus S^*$ invariant subspace corresponding to $\mcl$ is :
	$$\kcl =\{(K_0,k_1)\in \hdcr \times \hdcc:\langle K_0,z^nG_0\rangle _{\hdcr}+\langle k_1,z^ng\rangle _{\hdcc}=0~for~n\in \mathbb{N}\cup \{0\}\}.$$

\item In case $\wcl =\{0\}$, $$\mcl =\{F:F(z)=zk_1(z)G(z):k_1\in \kcl \}$$ with $S^*$ invariant subspace is $$\kcl =\{k_1\in \hdcc :\langle k_1,z^ng\rangle _{\hdcc}=0 ~for~n\in \mathbb{N}\cup\{0\}\}.$$ 
\end{enumerate}
\begin{rmrk}
	If $G=(g_1,g_2,\ldots ,g_m)$ is an arbitary element of $\hdcm$, then $\mcl=\hdcm \ominus \langle G\rangle$ will be of the form :
	\begin{enumerate}
		\item In the case when $\wcl \neq \{0\}$, $$\mcl =\{F:F(z)=F_0(z)K_0(z)+zk_1(z)G(z):(K_0,k_1)\in \kcl \subseteq \hdcr \times \hdcc\},$$ where the $S^*\oplus S^*$ invariant subspace corresponding to $\mcl$ is 
		$$\kcl =\{(K_0,k_1)\in \hdcr \times \hdcc:\langle F_0S^{*n}K_0+zS^{*n}k_1G, G \rangle =0~for~n\in \mathbb{N}\cup \{0\}\}.$$
		\item In case $\wcl =\{0\}$, $$\mcl =\{F:F(z)=zk_1(z)G(z):k_1\in \kcl \},$$ where the $S^*$ invariant subspace corresponding to $\mcl$ is  $$\kcl =\{k_1\in \hdcc :\langle zS^{*n}k_1G,G\rangle =0 ~for~n\in \mathbb{N}\cup\{0\}\}.$$ 
	\end{enumerate}
	
\end{rmrk}

Next we give some concrete examples through which we calculate the space $\kcl$ explicitly. To proceed further we need the following useful  result in the vector valued Hardy space $\hdcm$ setting.
\begin{ppsn}\label{c}
	Let $\Phi \in H^\infty(\D,\mathcal{L}(\mathbb{C}^m))\backslash\{0\} $ be of the form :
\begin{equation}
   \Phi=
   \begin{bmatrix}  
   \phi _1&0&\cdots&0\\
   0&\phi_2&\cdots&0\\
   \vdots&\vdots&\ddots&\vdots\\
   0&0& \cdots&\phi_m\\
   \end{bmatrix}_{m\times m,} 
\end{equation}
and let $\Theta$ be the inner part of $\Phi$. Then $Ker T_{\Phi ^*}=\mathcal{K}_{\Theta}=\hdcm\ominus \Theta \hdcm$. In particular if, 
\begin{equation}
\Psi=
\begin{bmatrix}  
\psi _1&0&\cdots&0\\
0&\psi_2&\cdots&0\\
\vdots&\vdots&\ddots&\vdots\\
0&0& \cdots&\psi_m\\
\end{bmatrix}_{m\times m,} \in H^\infty (\D,\mathcal{L}({\C}^m))
\end{equation}
with each $\psi _i$ is an outer function in $\hdcc$, then $T_{\Psi ^*}$ is an injective Toeplitz operator. 
	
\end{ppsn}
\begin{proof}
	Let $\Psi$  be of the above form with each $\psi _i $ is an outer function and consider the element $F=(f_1,f_2,\ldots ,f_m)\in KerT_{\Psi ^*}$. Then for all  $i\in \{1,2,\ldots ,m\},~P(\overline{\psi} _i f_i)=0$ which implies that $\overline{\psi} _i f_i \in \overline{H}^2_0$. Since each $\psi _i$ is outer, then $f_i \in \overline{H}^2_0 $ and hence each $f_i=0$. Therefore $F=0$ and hence $T_{\Psi ^*}$ is injective.\\
	
	Now using the cannonical factorization of each $\phi _i=\psi _i\theta _i$, where $\psi _i$ is outer and $\theta _i$ is inner, we have the following decomposition of $\Phi$:
	\begin{equation}
		\Phi=
	\begin{bmatrix}  
	\phi _1&0&\cdots&0\\
	0&\phi_2&\cdots&0\\
	\vdots&\vdots&\ddots&\vdots\\
	0&0& \cdots&\phi_m\\
	\end{bmatrix}=
	\begin{bmatrix}  
	\psi _1&0&\cdots&0\\
	0&\psi_2&\cdots&0\\
	\vdots&\vdots&\ddots&\vdots\\
	0&0& \cdots&\psi_m\\
	\end{bmatrix}\times
	\begin{bmatrix}  
	\theta _1&0&\cdots&0\\
	0&\theta_2&\cdots&0\\
	\vdots&\vdots&\ddots&\vdots\\
	0&0& \cdots&\theta_m \\
	\end{bmatrix}
	=\Psi \boldsymbol{\cdot}\Theta \text{ (say)}.
	\end{equation}
Therefore $\Theta$ is an inner multiplier and $\Psi$ is an outer function in the sense of V.I.Smirnov \cite{KK}. We call $\Theta$ as the inner part of $\Phi$ and $\Psi$ as the outer part of $\Phi$.
	Note that $KerT_{\Phi ^*}=\{F\in\hdcm :T_{\Phi ^*}(F)=0 \}=\{F\in \hdcm :T_{\Psi ^*}T_{\Theta ^*}(F)=0\}=\{F\in \hdcm :T_{\Theta ^*}(F)=0\}$~(as $T_{\Psi ^*} ~is~injective$) $=\mathcal{K}_\Theta$.

\end{proof}

\begin{xmpl}
	\begin{enumerate}[(i)]
		\item Let $G=(1,0,\ldots ,0)=k_0\otimes e_1$, then $\mcl =z\hdcm \oplus span\{e_2,e_3,\ldots ,e_m\}.$ Thus $\wcl =span\{e_2,e_3,\ldots ,e_m\}$ and hence we consider
		\begin{equation}
		F_0=
			F_0=
		\begin{bmatrix}  
		0&0&\cdots&0\\
		1&0&\cdots&0\\
		0&1&\cdots&0\\
		\vdots&\vdots&\ddots&\vdots\\
		0&0&\cdots&1 \\
		\end{bmatrix}_{m\times (m-1)}.
	\end{equation}
Therefore $$\mcl =\{F:F(z)=F_0(z)K_0(z)+k_1(z)\otimes e_1:(K_0,k_1) \in \kcl \subset  H^2_{\mathbb{C}^{m-1}}(\D) \times \hdcc \},$$ where $\kcl =\hdcm$  is a trivial $S^*$ invariant subspace.
\item 	
Let $G=\dfrac{1}{\sqrt{m}}(\theta _1,\theta _2,\ldots ,\theta_m)$ with each $\theta _i$ is a non constant inner function in $\hdcc$. Now consider 	
\begin{equation}
\Theta=\dfrac{1}{\sqrt{m}}
\begin{bmatrix}  
\theta _1&0&\cdots&0\\
0&\theta _2&\cdots&0\\
\vdots&\vdots&\ddots&\vdots\\
0&0&\cdots&\theta_m \\
\end{bmatrix}_{m\times m}. 
\end{equation}	
Then $\mcl =\hdcm \ominus \langle G \rangle =\mathcal{K}_{\Theta}\oplus z\Theta \hdcm $. Thus 	
\begin{equation}
F_0=
\begin{bsmallmatrix}  
C_{11}(1-\overline{\theta_1(0)}\theta_1) & C_{21}(1-\overline{\theta_1(0)}\theta_1)+C_{22}(-\overline{\theta_2(0)}\theta_1) &\cdots& C_{m1}(1-\overline{\theta_1(0)}\theta_1)+\ldots +C_{mm}(-\overline{\theta_m(0)}\theta_1) \\
C_{11}(-\overline{\theta_1(0)}\theta_2) & C_{21}(-\overline{\theta_1(0)}\theta_2)+C_{22}(1-\overline{\theta_2(0)}\theta_2) &\cdots&C_{m1}(-\overline{\theta_1(0)}\theta_2)+\ldots +C_{mm}(-\overline{\theta_m(0)}\theta_2) \\
\vdots&\vdots&\ddots&\vdots\\
C_{11}(-\overline{\theta_1(0)}\theta_m) & C_{21}(-\overline{\theta_1(0)}\theta_m)+C_{22}(-\overline{\theta_2(0)}\theta_m) &\cdots& C_{m1}(-\overline{\theta_1(0)}\theta_m)+\ldots +C_{mm}(1-\overline{\theta_m(0)}\theta_m) \\
\end{bsmallmatrix}_{m\times m}
\end{equation}	
and,
\begin{align}
G_0 &= \begin{bmatrix}
\overline{C}_{11}(\theta_1-m\theta_1(0)) \\
\overline{C}_{21}(\theta_1-m\theta_1(0))+\overline{C}_{22}(\theta_2-m\theta_2(0)) \\
\vdots \\
\overline{C}_{m1}(\theta_1-m\theta_1(0))+\ldots +\overline{C}_{mm}(\theta_m-m\theta_m(0))
\end{bmatrix} \in\hdcm, ~g=0.
\end{align}
Now using the Case(i), $\mathcal{M}$ has the following representation: $$\mcl =\{F:F(z)=F_0(z)K_0(z)+zk_1(z)G(z):(K_0,k_1) \in \kcl \subset  H^2_{\mathbb{C}^{m}}(\D) \times \hdcc \}$$ with the $S^*$ invariant subspace $\kcl =\mathcal{K}_\xi \times \hdcc$, where $\xi$ is the inner part of the following $H^\infty(\D,\mathcal{L}(\mathbb{C}^m))$ function: 
\begin{equation*}
\zeta =
\begin{bmatrix}  
\theta _1-m\theta_1(0)&0&\cdots&0\\
0&\sum_{i=1}^{2}\theta _i-m\theta_i(0)&\cdots&0\\
\vdots&\vdots&\ddots&\vdots\\
0&0&\cdots&\sum_{i=1}^{m}\theta _i-m\theta_i(0) \\
\end{bmatrix}_{m\times m}. 
\end{equation*}
Using Proposition \ref{c} one can easily check that $\kcl$ is $S^*$ invariant subspace of $\hdcmone$.
\item If we consider $G(z)=(\dfrac{1+z^k}{\sqrt{2}},0,\cdots ,0)~for~ k\geq 1$, then $$\mcl=\hdcm \ominus \langle G\rangle =span\{1-z^k,z,z^2,\ldots z^{k-1},z^{k+1},\ldots\}\oplus \hdcc \oplus \hdcc \oplus \ldots \oplus \hdcc \subseteq\hdcm .$$ Therefore 
\begin{equation*}
F_0=
\begin{bmatrix}  
\dfrac{1-z^k}{\sqrt{2}}&0&\cdots&0\\[1pt]
0&1&\cdots&0\\
\vdots&\vdots&\ddots&\vdots\\
0&0&\cdots&1\\
\end{bmatrix}_{m\times m}, 
\begin{aligned}
G_0 &= \begin{bmatrix}
\frac{z^k}{2}\\
0\\
\vdots \\
0\\
\end{bmatrix} \in\hdcm, \quad\text{and} \quad g(z)=\frac{z^{k-1}}{2}\in \hdcc.
\end{aligned} 
\end{equation*}

Thus by using Case(i), $$\mcl =\{F:F(z)=F_0(z)K_0(z)+zk_1(z)G(z):(K_0,k_1) \in \kcl \subset  H^2_{\mathbb{C}^{m}}(\D) \times \hdcc \}$$ with the $S^*$ invariant subspace $$\kcl =\{(K_0,k_1)\in \hdcm \times \hdcc : (S^*)^{k-1}k_1(z) =-(S^*)^{k}\langle K,k_z\otimes e_1\rangle ,K_0\in \hdcm \}.$$
\item Now consider $G(z)=\sqrt{\dfrac{1-|\alpha |^2}{m}}(k_\alpha (z),k_\alpha (z),\ldots ,k_\alpha (z))$, where $k_\alpha$ is a reproducing kernel at $\alpha \in \D\backslash \{0\}$ in $\hdcc$. Then $$\mcl =Ker T=\hdcm \ominus
\langle G\rangle =\{F\in \hdcm :\sum_{i=1}^{m}f_i(\alpha)=0\}.$$ Since $P_\mcl(k_0\otimes e_1)$ is non zero, then $\wcl$ is non trivial and hence  $F_0$ has the same form as in (2.3) with each $g_i(z)=k_\alpha (z)$. In this case $G_0=0,~ g(z)=m\overline{\alpha}k_\alpha (z)$ which is an outer function. Moreover by using Case(i), $$\mcl =\{F:F(z)=F_0(z)K_0(z)+zk_1(z)G(z):(K_0,k_1) \in \kcl \subset  H^2_{\mathbb{C}^{r}}(\D) \times \hdcc \}$$ with the $S^*$ invariant subspace
$$\kcl =H^2_{\mathbb{C}^{r}}(\D) \times \{0\}.$$	
	\end{enumerate}
\end{xmpl}

Next we consider the case when $\Phi \in L^\infty (\mathbb{T},\mathcal{L}(\C ^m,\C^m))$ is non zero almost everywhere on $\mathbb{T}$ and with this assumption we consider three important subcases. To proceed further we need the following useful results :
\begin{thm}\label{d}
	For $\Phi , \Psi \in L^\infty (\mathbb{T},\mathcal{L}(\C ^m,\C^m))$; if either $\Psi ^* \in H^\infty _{\mathcal{L}(\C^m,\C^m)}(\D)~or ~\Phi \in H^\infty _{\mathcal{L}(\C^m,\C^m)}(\D)$, then $T_\Psi T_\Phi$ is a Toeplitz operator; in both cases $T_\Psi T_\Phi =T_{\Psi \Phi}$. 
	\end{thm} 
The proof of the above theorem follows similarly as in the scalar valued case and hence left it to the reader. On the other hand it is important to observe that the converse of the theorem is not true in general. For example consider 
 \begin{equation}
 \Psi=
 \begin{bmatrix}  
e^{i\theta}&0&\cdots &0\\
0&0&\cdots&0\\
\vdots&\vdots&\ddots&\vdots\\
0&0&\cdots&0\\
\end{bmatrix}_{m\times m,}\text{, and }
\Phi=
\begin{bmatrix}  
0&0&\cdots &0\\
0&e^{-i\theta}&\cdots&0\\
\vdots&\vdots&\ddots&\vdots\\
0&0&\cdots&0\\
\end{bmatrix}_{m\times m}.
 \end{equation}
Then it is easy to observe that $\Psi \in H^\infty _{\mathcal{L}(\C^m,\C^m)}(\D)$ and $\Phi ^*\in H^\infty _{\mathcal{L}(\C^m,\C^m)}(\D)$ but $T_\Psi T_\Phi =0$ is a Toeplitz operator.
Now we denote 
 \begin{equation}
 {\bf{Z}}=
\begin{bmatrix}  
z&0&\cdots &0\\
0&z&\cdots&0\\
\vdots&\vdots&\ddots&\vdots\\
0&0&\cdots&z\\
\end{bmatrix}_{m\times m,} \in H^\infty _{\mathcal{L}(\C^m,\C^m)}(\D)  \quad \text{and hence}\quad T_{\bf{Z}} ^{*}=T_{{\bf{Z}^{^*}}}.
 \end{equation}
Thus  $$T_{{\bf{Z}}^{^*}}T_\Phi =T_{\bf{Z}^{^*}\Phi}=T_{\Phi \bf{Z}^{^*}}~,~\forall~\Phi \in L^\infty (\mathbb{T},\mathcal{L}(\C ^m,\C^m)).$$

Next we provide an equivalent condition on the element $F$ to be in the kernel of $T_2=T_\Phi +\sum\limits_{i=1}^{2}\langle .,G_i\rangle H_i$. To do that suppose $F \in KerT_2$ with $F(0)=0$. Then the following equivalent conditions hold.


\begin{align}
\nonumber & ~T_2(F)=0\\
\label{e}\Leftrightarrow ~&~T_\Phi (F) +\sum_{i=1}^{2}\langle F,G_i\rangle H_i =0\\
\nonumber \Leftrightarrow ~&~P_m\big(\Phi F +\sum_{i=1}^{2}\langle F,G_i\rangle H_i )=0\\
\nonumber \Leftrightarrow ~&~\Phi F +\sum_{i=1}^{2}\langle F,G_i\rangle H_i \in \overline{H}^2_0.
\end{align}
Applying $T_{\bf{Z}}^*$ on both sides of \ref{e} and using Theorem \ref{d}, we have the following equivalent conditions.
\begin{align}
\nonumber & ~T_{\Phi\bf Z^*}(F) + \sum_{i=1}^{2}\langle F,G_i\rangle S^*H_i =0\\
\nonumber \Leftrightarrow ~&~P_m\big(\Phi \dfrac{F}{z} +\sum_{i=1}^{2}\langle F,G_i\rangle S^*H_i )=0\\
\label{f}\Leftrightarrow ~&~T_{\Phi}\big( \dfrac{F}{z}) +\sum_{i=1}^{2}\langle F,G_i\rangle S^*H_i =0\\
\label{fi}\Leftrightarrow ~& ~\Phi \dfrac{F}{z} +\sum_{i=1}^{2}\langle F,G_i\rangle S^*H_i \in \overline{H}^2_0.
\end{align}
Now if we recapitulate our problem, actually we have to show the kernel of  $T_2$ is nearly $S^*$ invariant with finite defect and to do so we have to find a vector $V$ in some appropiate finite dimesional subspace $\fcl$ such that $$S^*F +V \in Ker T_2 \quad \text{with} \quad F\in KerT_2,~F(0)=0.$$ which is equivalent to the following equations 
\begin{align}
\nonumber & ~T_2(S^*F+V)=0\\
\label{g}\Leftrightarrow ~&~T_\Phi (\dfrac{F}{z}+V) +\sum_{i=1}^{2}\langle \dfrac{F}{z}+V,G_i\rangle H_i =0\\
\nonumber \Leftrightarrow ~&~P_m\big(\Phi (\dfrac{F}{z}+V)+\sum_{i=1}^{2}\langle \dfrac{F}{z}+V,G_i\rangle H_i )=0\\
\label{gi}\Leftrightarrow ~&~P_m\big(\Phi V-\sum_{i=1}^{2}\langle F,G_i\rangle S^*H_i+\sum_{i=1}^{2}\langle \dfrac{F}{z}+V,G_i\rangle H_i )=0,\text{(using ~\ref{fi}~)}\\
\label{hi}\Leftrightarrow ~&~\Phi V-\sum_{i=1}^{2}\langle F,G_i\rangle S^*H_i+\sum_{i=1}^{2}\langle \dfrac{F}{z}+V,G_i\rangle H_i  \in \overline{H}^2_0.
\end{align}
In the following three sections we are going to show that kernel of  $T_2$ is nearly $S^*$ invariant with finite defect in various important cases mentioned by Liang and Partington in \cite{YP}. Furthermore we also calculate the defect space $\mathcal{F}$ explicitly in those mentioned cases.

\section{Kernel of finte rank perturbation of Toeplitz operator having symbol an inner multiplier}
In this section we consider $\Phi =\Theta$ is an inner function. Then $T_\Theta $ is a nice Toeplitz operator. Our main aim is to show that the kernel of $T_n=T_{\Theta}+\sum\limits_{i=1}^{n}\langle .,G_i\rangle H_i$ is nearly $S^*$ invariant with finite defect and to calculate  the defect space explicitly. To avoid complicacies in the calculations, we restrict our self in the case $n=2$ and finally we state our result in general setting.  For this purpose let us consider $F\in KerT_2$ with $F(0)=0.$ Then the equation (2.17) can be rewritten as, $$\Theta F+\sum_{i=1}^{2}\langle F,G_i\rangle H_i =0.$$ In this context the equation \ref{f} becomes
\begin{align}
\label{h}\Theta \dfrac{F}{z} +\sum_{i=1}^{2}\langle F,G_i\rangle S^*H_i =0.
\end{align}
Next by acting $T_\Theta ^*$ on both sides of \ref{h} we have $$T_{\Theta ^*}\big(\sum_{i=1}^{2}\langle F,G_i)H_i (=\Theta ^*\big(\sum_{i=1}^{2}\langle F,G_i)\rangle S^*H_i)=-\dfrac{F}{z} \in \hdcm .$$ 
Therefore,\ref{g} becomes $$\Theta (\dfrac{F}{z}+V)+\sum_{i=1}^{2}\langle \dfrac{F}{z} +V,G_i \rangle H_i=0.$$ Using \ref{h}, the above equation is equivalent to
\begin{align}
\label{i}\Theta V -\sum_{i=1}^{2}\langle F,G_i\rangle S^*H_i+ \sum_{i=1}^{2}\langle \dfrac{F}{z}+V,G_i\rangle H_i=0 .
\end{align}
Now consider those $V\in \hdcm $ for which $$\Theta V=\sum_{i=1}^{2}\langle F,G_i \rangle S^*H_i \in \hdcm .$$ The above yields that  
\begin{align}\label{j}
V=\sum_{i=1}^{2}\langle F,G_i \rangle T_{\Theta ^*}(S^*H_i)=\sum_{i=1}^{2}\langle F,G_i \rangle S^*(T_{\Theta ^*}(H_i)),
\end{align}
because $$T_{\Theta ^*}S^*=T_{\Theta ^*}T_{\overline{\bf z}}=T_{\overline{\bf z}}T_{\Theta ^*}=S^*T_{\Theta ^*}.$$
Then the left hand side of \ref{i} becomes 
\begin{align*}
&\sum_{i=1}^{2}\langle F,G_i \rangle S^*H_i -\sum_{i=1}^{2}\langle F,G_i \rangle S^*H_i +\sum_{i=1}^{2}\langle \dfrac{F}{z}+V,G_i\rangle H_i \\
&=\sum_{i=1}^{2}\langle \Theta (\dfrac{F}{z}+V),\Theta G_i\rangle H_i ~~ (\text{since} \quad\Theta \quad \text{is inner})\\
&=\sum_{i=1}^{2}\langle \Theta \dfrac{F}{z}+\sum_{i=1}^{2}\langle F,\Theta G_i \rangle S^*H_i,G_i\rangle H_i =0.
\end{align*}
Using the above construction of $V$ in \ref{j} we define the defect space $$\fcl =\bigvee _{i=1}^2\{T_{\Theta ^*}(S^*H_i)\}=\bigvee _{i=1}^2\{S^*(T_{\Theta ^*}H_i)\},$$ with dimension at most 2. Hence $KerT_2$ is nearly $S^*$ invariant with defect atmost 2. Therefore by repeating the above calculations again we have the following theorem regarding the kernel of $T_n$.
\begin{thm}
	If $\Phi =\Theta$ is an inner multiplier, then the subspace $Ker T_n$ is  nearly $S^*$ invariant subspace with defect atmost $n$ and the defect space is $$\fcl =\bigvee _{i=1}^n\{T_{\Theta ^*}(S^*H_i)\}=\bigvee _{i=1}^n\{S^*(T_{\Theta ^*}H_i)\}.$$
\end{thm}
Next we give some simple example regarding this case.
\begin{xmpl}
	If $\Phi(z)= \begin{bmatrix}  
	z^p &0 &\cdots &0\\[1pt]
	0 &z^p &\cdots &0\\
	\vdots&\vdots&\ddots&\vdots\\
	0 &0 &\cdots&z^p \\
	\end{bmatrix}_{m\times m,},p\in \mathbb{N}$ and $n=1$. Then $KerT_1$ is nearly $S^*$ invariant subspace with defect $1$ and the defect space is $\fcl =\langle (S^*)^{p+1}H_1\rangle $.
\end{xmpl}
Now coming back to the application part of Theorem \ref{a}, as discussed earlier we deal with rank one perturbation of Toeplitz operator. Thus  $$T=T_\Theta +\langle .,G\rangle H~,~with~ \norm G _2 =1 \quad \text{and} \quad S^*H\neq 0.$$ Therefore $$\mcl =Ker T\subseteq \langle \Theta ^*H \rangle ~~and~~\fcl =\langle S^*(T_{\Theta ^*}H)\rangle =\langle S^*(\Theta ^*H)\rangle. $$ Next we consider $F= \mu \Theta ^*H \in \mcl$ satisfying $T(F)=0$ which is equivalent to $$\mu (1+\langle \Theta ^*H,G\rangle)=0 .$$
Therefore we have the following two cases to consider:\\
Case 1. If $1+\langle \Theta ^*H,G\rangle \neq 0$, then $\mcl =\{0\}$ is a trivial $S^*$ invariant subspace.\\
Case 2. If $1+\langle \Theta ^*H,G\rangle = 0$, then it yields $\mcl= \langle \Theta ^*H \rangle $. Let $\{A_k\}_{k=0}^\infty$ be the coefficients of the Taylor series expansion of $\Theta ^* H$ (since $\Theta ^* H \in \hdcm $) and we calculate the subspace $\kcl$ in two subcases.
\begin{enumerate}[(i)]
	\item If $A_0 \neq 0,$ then there exists at least one function in $\mcl$ that do not vanish at $0$. Now $\wcl =\mcl \ominus (\mcl \cap z\hdcm)$ has dimension $1$ (since dim $ \mcl =1$). Therefore exactly one $F_i\neq 0$ for some $i(1\leq i\leq m)$ which generates $\wcl$. On the other hand $F_i=P_\mcl(k_0\otimes e_i)=\dfrac{\langle k_0\otimes e_i, \Theta ^*H\rangle \Theta ^*H}{\norm {\Theta ^*H}_2^2}.$ and hence 
	$F_0=
	\begin{bmatrix}
	C_{11}F_i 
	\end{bmatrix}_{m\times 1}.$ Thus\\
$\mcl = \langle \Theta ^*H\rangle =\{F:F=F_0K_0+zk_1\dfrac{S^*(\Theta ^*H)}{\norm {S^*(\Theta ^* H)}}:(K_0,k_1)\in \kcl \subset \hdcc \times \hdcc\}$\\
$=\{F:F=K_0C_{11}\dfrac{\langle k_0\otimes e_i, \Theta ^*H\rangle \Theta ^*H}{\norm {\Theta ^*H}_2^2}+\dfrac{k_1(\Theta ^*H -A_0)}{\norm{S^*(\Theta ^*H)}}:(K_0,k_1)\in \kcl \subset \hdcc  \times \hdcc\}$.\\
Note that the structure of $\mcl$ forces to conclude that $K_0\in \mathbb{C}$ and $k_1(\Theta ^*H -A_0)= \xi \Theta ^*H$ with $\xi \in \hdcc$ which will be valid if and only if $K_0 \in \mathbb{C}$ and $k_1=0$. Therefore the required nearly $S^*$ invariant subspace of finite defect is $$\mcl=\{F:F=K_0C_{11}\dfrac{\langle k_0\otimes e_i, \Theta ^*H\rangle }{\norm {\Theta ^*H}_2^2}\Theta ^*H :(K_0,0)\in \kcl\subset \hdcc \times \hdcc \}$$
and the corresponding $S^*\oplus S^*$ invariant subspace is $\kcl =\mathbb{C}\times \{0\}$ of $H^2_{\mathbb{C}^2}(\D)$.
\item If $A_0=0$, then $\wcl =\{0\}$ and hence $F_0=0$. In that case the nearly $S^*$ invariant subspace of finite defect is  $$\mcl =\langle \Theta ^*H \rangle =\{F: F=k_1\dfrac{\Theta ^*H }{\norm{S^*(\Theta ^*H)}}:k_1\in \kcl \subset \hdcc \}$$ 
with the corresponding $S^*$ invariant subspace is $\kcl =\C$ of $\hdcc$.
\end{enumerate}
\section{Kernel of finte rank perturbation of Toeplitz operator with symbol having factorization in $\mathcal{G}H^\infty(\mathcal{L}(\C ^m))$}

In this section we deal with  very special type of $\Phi \in L^\infty
(T,\mathcal{L}(\C ^m,\C^m))$ such that $\Phi= F_1^*F_2$, with $F_j \in \mathcal{G}H^\infty (\mathcal{L}(\C ^m,\C^m))$  for $j=1,2$. Here $\mathcal{G}H^\infty(\mathcal{L}(\C ^m,\C^m))$ denotes the set of all invertible elements in $H^\infty_{\mathcal{L}(\C ^m,\C^m)}(\D)$. For more details in the literature we refer to \cite{BARK} and the references cited therein.
Note that for any vector $F\in Ker T_2$ with $F(0)=0$, (2.17) can be rewritten as
\begin{align}\label{k}
T_{F_1^*}T_{F_2^*}(F) +\sum_{i=1}^{2}\langle F,G_i\rangle H_i =0. 
\end{align}
 
Since $F(0)=0$ and $T_{\overline{\bf z}}T_{F^*_1}=T_{F^*_1}T_{\overline{\bf z}}$, then by using these fact along with  the action of $T_{\bf{Z}^{^*}}$ on both sides of \ref{k} we get
\begin{align}\label{l}
T_{F_1^*}(F_2\dfrac{F}{z}) +\sum_{i=1}^{2}\langle F,G_i\rangle S^*H_i =0. 
\end{align}
Now by applying $T_{F_2^{-1}}T_{F_1^{*-1}}$  on both sides of \ref{l} we have
\begin{align}\label{m}
\dfrac{F}{z} +\sum_{i=1}^{2}\langle F,G_i\rangle T_{F_2^{-1}}T_{F_1^{*-1}}(S^*H_i) =0. 
\end{align}
Thus by using \ref{l} our desired result \ref{g} is equivalent to \
\begin{align}\label{n}
T_{F^*_1}T_{F_2}(V)-\sum_{i=1}^{2}\langle F,G_i\rangle S^*H_i +\sum_{i=1}^{2}\langle \dfrac{F}{z}+V,G_i \rangle H_i =0.
\end{align}
Next we consider $V$ such that 
\begin{align*}
& T_{F^*_1}(F_2V)=\sum_{i=1}^{2} \langle F,G_i\rangle S^*H_i \\
 \Leftrightarrow ~&~ V= \sum_{i=1}^{2} \langle F,G_i \rangle T_{F_2^{-1}} T_{F_1^{*-1}}(S^*H_i).
\end{align*}
If we consider the above $V$, then by using \ref{m} the left hand side of \ref{n} becomes
\begin{align}
\sum_{i=1}^{2}\langle F,G_i\rangle S^*H_i-\sum_{i=1}^{2}\langle F,G_i\rangle S^*H_i +\sum_{i=1}^{2}\langle \dfrac{F}{z}+\sum_{i=1}^{2} \langle F,G_i \rangle T_{F_2^{-1}} T_{F_1^{*-1}}(S^*H_i),G_i \rangle H_i =0 . 
\end{align}
Moreover from the choice of $V$ it is clear that the defect space should be $$\fcl =\bigvee _{i=1}^2 \{F_2^{-1}S^*(T_{F_1^{*-1}}H_i)\}$$ with dimension atmost 2. Therefore in general setting we have the following theorem regarding the kernel of $T_n$:
\begin{thm}
	Let $\Phi =F_1^*F_2$ with $F_1,F_2\in \mathcal{G}H^\infty(\mathcal{L}(\C ^m)).$ Then the subspace $Ker T_n$ is nearly $S^*$ invariant with defect atmost $n$ and the defect space is $$\fcl =\bigvee _{i=1}^n\{F_2^{-1}T_{F_1^{*-1}}(S^* H_i)\}=\bigvee _{i=1}^n \{F_2^{-1}S^*(T_{F_1^{*-1}}H_i)\}.$$
\end{thm}
\begin{rmrk}
	\begin{enumerate}
		\item If we suppose $\Phi ^* \in \mathcal{G}H^\infty(\mathcal{L}(\C ^m))$ , then the kernel of $T_n$ is nearly $S^*$ invariant with defect atmost $n$ and the defect space is $$\fcl =\bigvee _{i=1}^n \{S^*(T_{\Phi ^{-1}}H_i)\}.$$
	   \item If we consider $\Phi \in \mathcal{G}H^\infty(\mathcal{L}(\C ^m))$, then the kernel of $T_n$ is also nearly $S^*$ invariant with defect atmost $n$ and the defect space will be $$\fcl =\bigvee _{i=1}^n \{T_{\Phi ^{-1}}S^*(H_i)\} .$$
	
	\end{enumerate}
\end{rmrk}
For simplicity if we assume $n=1$ and $\Phi \in \mathcal{G}H^\infty (\mathcal{L}(\C ^m))$, then we have the following corollary.
\begin{crlre}
	For $n=1$ and $\Phi \in \mathcal{G}H^\infty (\mathcal{L}(\C ^m))$, the subspace $Ker T_1 ( i.e.,Ker T)$ is nearly $S^*$ invariant with defect atmost $1$ and the defect space is $\fcl =\langle \dfrac{S^*H}{\Phi} \rangle$.
\end{crlre}
We now discuss about the application part of Theorem(2.2). Since in this section we consider that $\Phi =F_1^*F_2$, where $F_1,F_2\in \in \mathcal{G}H^\infty (\mathcal{L}(\C ^m))$ almost everywhere on $\mathbb{T}$, then $$\mcl =Ker T\subseteq \langle F_2^{-1}(T_{F_1^{*-1}}H)\rangle$$ and the defect space is $$\fcl =\langle F_2^{-1}T_{F_1^{*-1}}(S^*H) \rangle =\langle F_2^{-1}S^*(T_{F_1^{*-1}}H) \rangle .$$
Let us consider any vector $F=\lambda F_2^{-1}(T_{F_1^{*-1}}H)\in \mcl$ satisfying $T(F)=0$. Then it is equivalent to the following
\begin{align}\label{o}
\lambda (1+\langle F_2^{-1}(T_{F_1^{*-1}}H),G\rangle )=0.
\end{align} 
Therefore according to \ref{o} we have the following two cases.\\ 
Case 1: If $1+\langle F_2^{-1}(T_{F_1^{*-1}}H),G\rangle \neq 0$, then it yields that $\mcl =\{0\}$ which is a trivial $S^*$ invariant subspace.\\
Case 2: If $1+\langle F_2^{-1}(T_{F_1^{*-1}}H),G\rangle = 0$, then $\mcl= \langle F_2^{-1}(T_{F_1^{*-1}}H)\rangle$. \\
Let $\{A_k\}_{k=0}^\infty $ be the Taylor coefficients of $T_{F_1^{*-1}}H$ and let $\{\Phi _k\}_{k=0}^\infty$ be the Taylor coefficients of $F_2^{-1}$.Therefore it follows that
\begin{align*}
[F_2^{-1}(T_{F_1^{*-1}}H)](z)=\sum_{n=0}^{\infty}(\sum_{k=0}^{n}\Phi _kA_{n-k})z^n, \\
zS^*(T_{F_1^{*-1}}H)(z)=(T_{F_1^{*-1}}H)(z)-A_0.
\end{align*}
Depending upon the context we again divide our analysis in two subcases to calculate $\kcl$.
\begin{enumerate}
	\item If $F_2 ^{-1}(T_{F_1^{*-1}}H)(0)\neq 0$, then $\wcl$ has dimension exactly $1$ and therefore only one $F_i$ generates $\wcl$. Thus $$F_i=P_{\mcl}(k_0 \otimes e_i)= \dfrac{\langle k_0\otimes e_i,F_2^{-1}(T_{F_1^{*-1}}H)\rangle F_2^{-1}(T_{F_1^{*-1}}H)}{\norm {F_2^{-1}(T_{F_1^{*-1}}H}^2}$$ and hence $$F_0=[C_{11}F_2^{-1}(T_{F_1^{*-1}}H)]_{m\times 1}\quad \text{and} \quad E(z)=\dfrac{F_2^{-1}S^*(T_{F_1^{*-1}}H)}{\norm {F_2^{-1}S^*(T_{F_1^{*-1}}H)}^2}.$$ Therefore from the first case of Theorem\ref{a} it follows that $$\mcl =\langle F_2^{-1}(T_{F_1^{*-1}}H)\rangle =\{F:F=F_0K_0+zk_1E_1 :(K_0,k_1)\in\kcl\subset \hdcc\times\hdcc \}.$$ Following the above identities we have $$K_0 \in \C \quad \text{and} \quad k_1\dfrac{F_2^{-1}(T_{F_1^{*-1}}H -A_0) }{\norm {F_2^{-1}S^*(T_{F_1^{*-1}}H)}^2}=\xi F_2^{-1}(T_{F_1^{*-1}}H) \quad \text{with} \quad \xi \in \C ,$$ which will be true if and only if $K_0 \in \C$ and $k_1=0.$ Thus the subspace $\mcl$ has the following representation $$\mcl =\{F:F=C_{11}K_0 \dfrac{F_2^{-1}(T_{F_1^{*-1}}H)}{\norm {F_2^{-1}(T_{F_1^{*-1}}H)}^2}:(K_0,0)\in \kcl \}$$
with $\kcl= \C \times \{0\}$ is a $S^*\oplus S^*$ invariant subspace of $H^2(\D,\C ^2)$.
\item If $F_2^{-1}(T_{F_1^{*-1}}H)(0)=\Phi_0A_0=0$, then $\wcl =\{0\}$ and hence from the case(ii) of Theorem \ref{a} we have
$$\mcl =\langle F_2^{-1}(T_{F_1^{*-1}}H) \rangle =\{F:F=k_1\dfrac{F_2^{-1}(T_{F_1^{*-1}}H-A_0)}{\norm {F_2^{-1}S^*(T_{F_1^{*-1}}H)}}:k_1\in\kcl \subset \hdcc \}$$
which will be true if and only if $A_0 =0$ and $k_1 \in \C$. Thus we have the subspace $$\mcl =\{F:F=k_1\dfrac{F_2^{-1}(T_{F_1^{*-1}}H)}{\norm{F_2^{-1}S^*(T_{F_1^{*-1}}H)}}:k_1\in \kcl  \}$$ with $\kcl =\C$ an $S^*$ invariant subspace of $\hdcc$.

\end{enumerate}

\section{Kernel of finte rank perturbation of Toeplitz operator having symbol adjoint of an inner multiplier}
In this section we consider $\Phi(z) =\Theta ^*(z) $, where $\Theta$ is a non constant inner multiplier. It is well known that the kernel of the Toeplitz operator $T_{\Theta ^*}$ is the model space $\mathcal{K}_\Theta$. In other words $Ker T_{\Theta ^*}=\mathcal{K}_\Theta =\hdcm \ominus \Theta \hdcm$. 
Like in the previous section first we find an equivalent condition \ref{fi} for any vector $F\in ker T_2$ with $F(0)=0$. Note that
\begin{align*}
\Theta ^*\dfrac{F}{z}+ \sum_{i=1}^{2}\langle F,G_i \rangle S^*H_i \in \overline{H}_0^2
\end{align*}
is equivalent to 
\begin{align}\label{p}
\dfrac{F}{z}+ \sum_{i=1}^{2}\langle F,G_i \rangle \Theta S^*H_i \in \Theta \overline{H}_0^2.
\end{align}
Therefore in this case the equations\ref{gi} and \ref{hi} changed into
\begin{align}
\label{q}~&~P_m\left(\Theta ^* V-\sum_{i=1}^{2}\langle F,G_i\rangle S^*H_i+\sum_{i=1}^{2}\langle \dfrac{F}{z}+V,G_i\rangle H_i \right)=0.\\
\nonumber \Leftrightarrow ~&~\Theta ^* V-\sum_{i=1}^{2}\langle F,G_i\rangle S^*H_i+\sum_{i=1}^{2}\langle \dfrac{F}{z}+V,G_i\rangle H_i  \in \overline{H}^2_0.
\end{align}
Next we consider three cases for the construction of the defect space $\fcl$ in this case.\\

     Case 1: Suppose $G_1\in \Theta \hdcm$ and  $G_2 \in \Theta \hdcm$.\\
Due to condition \ref{p}, it follows that 
\begin{align}
\label{r}\left\langle \dfrac{F}{z}+ \sum_{i=1}^{2}\langle F,G_i \rangle \Theta S^*H_i ,G_1 \right\rangle =0.
\end{align}
and
\begin{align}
\label{s}\left\langle \dfrac{F}{z}+ \sum_{i=1}^{2}\langle F,G_i \rangle \Theta S^*H_i ,G_2 \right\rangle =0.
\end{align}  
Now we consider $V=\sum\limits_{j=1}^{2}\langle F,G_j \rangle \Theta S^*H_j$. Then substituting this $V$ and using the above two identities \ref{r} and \ref{s}, the left hand side of \ref{q} becomes
$$P_m\left (\Theta ^*\left(\sum_{i=1}^{2}\langle F,G_i \rangle  \Theta S^*H_i\right) -\sum_{i=1}^{2}\langle F,G_i \rangle  S^*H_i +\sum_{i=1}^{2}\left\langle \dfrac{F}{z}+\sum_{j=1}^{2}\langle F,G_j \rangle \Theta S^*H_j,G_i\right\rangle H_i \right )=0. $$
Thus we define the defect space as 
\begin{align}\label{rs}
\fcl =\bigvee _{i=1}^2\left\{\Theta S^*H_i \right\}
\end{align}
 and hence $KerT_2$ is nearly $S^*$ invariant with defect atmost $2$.\\

  Case 2: In this case we consider  $G_1\in \Theta \hdcm$ and  $G_2 \notin \Theta \hdcm$ (On a similar note, one can assume that $G_1\notin \Theta \hdcm$ and $G_2 \in \Theta \hdcm$, then the corresponding analysis will be of similar kind). First we note that \ref{r} holds in this case. Now if the identity \ref{s} also hold in this case, then the defect space will be exactly same as in Case 1. Therefore we assume that the identity \ref{s} does not hold in this case, then the vector $G_2$ can be decomposed in the following way
  $$G_2=U+W \quad \text{with} \quad U\neq 0\in \mathcal{K}_\Theta =Ker T_{\Theta ^*}\quad \text{and} \quad W\in \Theta \hdcm . $$
Thus 
\begin{align}\label{t}
T_{\Theta ^*}(U)=P_m(\Theta ^* U)=0 .
\end{align}

Now we choose $V=\sum\limits_{j=1}^{2}\langle F,G_j\rangle \Theta S^*H_j +\xi _F U, $ where $\xi _F$ is a constant which satisfies the following identity
\begin{align}\label{u}
\xi _F\norm {U}^2=-\left\langle \dfrac{F}{z}+\sum_{j=1}^{2}\langle F,G_j\rangle \Theta S^*H_j,G_2 \right\rangle \neq 0. 
\end{align} 
If we substitute $V$ in the left hand side of \ref{q}, then we have 
\begin{align*}
&P_m\Bigg(\Theta ^* (\sum_{j=1}^{2}\langle F,G_j\rangle \Theta S^*H_j +\xi _F U)-\sum_{i=1}^{2}\langle F,G_i\rangle S^*H_i\\
&\hspace*{2.5in}+\sum_{i=1}^{2}\left\langle \dfrac{F}{z}+(\sum_{j=1}^{2}\langle F,G_j\rangle \Theta S^*H_j +\xi _F U),G_i\right\rangle H_i \Bigg)\\
=&P_m \left( \xi _F\Theta ^*U+\sum_{i=1}^{2}\left\langle \dfrac{F}{z}+(\sum_{j=1}^{2}\langle F,G_j\rangle \Theta S^*H_j +\xi _F U),G_i\right \rangle H_i \right) \\
=&\sum_{i=1}^{2}\left\langle \dfrac{F}{z}+(\sum_{j=1}^{2}\langle F,G_j\rangle \Theta S^*H_j +\xi _F U),G_i\right \rangle H_i~~~\text{(using~\ref{t})}\\
=&\left\langle \dfrac{F}{z}+(\sum_{j=1}^{2}\langle F,G_j\rangle \Theta S^*H_j +\xi _F U),G_2\right \rangle H_2~~~~\text{(using~\ref{r}~and~}\langle U,G_1\rangle =0)\\
=&\left(\left\langle \dfrac{F}{z}+\sum_{j=1}^{2}\langle F,G_j\rangle \Theta S^*H_j ,G_2\right \rangle +\xi_F \norm {U}^2 \right ) H_2 =0. ~~~~~~\text{(using ~\ref{u})}
\end{align*}

Therefore by using the construction of $V$ we define the defect space as
\begin{align}\label{v}
\fcl =\bigvee\{\Theta S^*H_1,\Theta S^*H_2,P_{K_\Theta}G_2\},
\end{align}
where $P_{\mathcal{K}_\Theta}:L^2 \to \mathcal{K}_\Theta$ is an orthogonal projection onto $\mathcal{K}_\Theta$. Therefore $Ker T_2$ is  nearly $S^*$-invariant with defect atmost $3$. The defect space is either given by \ref{v} or given by \ref{rs} depending upon the condition \ref{s}.

Case 3: In this case we consider that none of $G_1,G_2$ is in $\Theta \hdcm$ , that is $G_1 \notin \Theta \hdcm$ and $G_2\notin \Theta \hdcm$.
If both \ref{r} and \ref{s} hold, then the defect space will be same as in Case 1, and if either  \ref{r} or \ref{s} hold, then the defect space is same as in Case 2. 
Therefore we assume that
\begin{align}
\left\langle \dfrac{F}{z}+ \sum_{i=1}^{2}\langle F,G_i \rangle \Theta S^*H_i ,G_1 \right\rangle \neq 0,\\
\left\langle \dfrac{F}{z}+ \sum_{i=1}^{2}\langle F,G_i \rangle \Theta S^*H_i ,G_2 \right\rangle \neq 0.
\end{align}  
In that case we decompose the vectors $G_1,G_2$ in the following way $$G_1=X_1+Y_1\quad \text{and} \quad G_2=X_2+Y_2,$$ where $X_i(\neq 0)\in \mathcal{K}_\Theta =KerT_{\Theta ^*}$ and $Y_i\in \Theta \hdcm$ for $i=1,2$. Therefore 
\begin{align}\label{w}
T_{\Theta ^*}(X_j)=P_m(\Theta ^*X_j)=0~~for~~j=1,2.
\end{align}
Now we choose rank -2 operator in such a way so that
\begin{align}\label{x}
\langle G_1,G_2\rangle =0 \quad \text{with}\quad \langle P_{K_\Theta}G_1,P_{K_\Theta}G_2 \rangle =\langle X_1,X_2\rangle =0.
\end{align}
Recall that our desired result is \ref{q} and for that purpose we choose $V$ such that $$V=\sum_{k=1}^{2}\langle F,G_k\rangle \Theta S^*H_k+\xi _{F,1}X_1 + \xi _{F,2}X_2,$$ 
where
\begin{align}\label{y}
 \xi _{F,j}\norm{X_j}^2=-\left\langle \dfrac{F}{z}+\sum_{k=1}^{2}\langle F,G_k\rangle \Theta S^*H_k, G_j \right \rangle (\neq 0),~j=1,2.
\end{align}

Substituting the above $V$ in the left hand side of \ref{q} we get
\begin{align*}
&P_m\left(\Theta ^* (\sum_{k=1}^{2}\langle F,G_k\rangle \Theta S^*H_k+\xi _{F,1}X_1 + \xi _{F,2}X_2)-\sum_{i=1}^{2}\langle F,G_i\rangle S^*H_i+\sum_{i=1}^{2}\langle \dfrac{F}{z}+V,G_i\rangle H_i \right)\\
=&P_m\left( \xi _{F,1}\Theta ^*X_1 + \xi _{F,2}\Theta ^*X_2+ \sum_{i=1}^{2}\langle \dfrac{F}{z}+V,G_i\rangle H_i \right)\\
=&\sum_{i=1}^{2}\langle \dfrac{F}{z}+V,G_i\rangle H_i~\text{(using ~\ref{w})} \\
=&\sum_{i=1}^{2}\langle \dfrac{F}{z}+\sum_{k=1}^{2}\langle F,G_k\rangle \Theta S^*H_k+\xi _{F,1}X_1 + \xi _{F,2}X_2,G_i\rangle H_i \\
=&\left\langle \dfrac{F}{z}+\sum_{k=1}^{2}\langle F,G_k\rangle \Theta S^*H_k+\xi _{F,1}X_1 + \xi _{F,2}X_2,G_1\right\rangle H_1 \\
&\hspace*{.5in}+\left\langle \dfrac{F}{z}+\sum_{k=1}^{2}\langle F,G_k\rangle \Theta S^*H_k+\xi _{F,1}X_1 + \xi _{F,2}X_2,G_2\right\rangle H_2\\
=&\left\langle \dfrac{F}{z}+\sum_{k=1}^{2}\langle F,G_k\rangle \Theta S^*H_k ,G_1 \right\rangle H_1+\langle \xi _{F,1}X_1+ \xi _{F,2}X_2,X_1+Y_1\rangle H_1 \\ 
&\hspace*{.5in}+\left\langle \dfrac{F}{z}+\sum_{k=1}^{2}\langle F,G_k\rangle \Theta S^*H_k ,G_2\right\rangle H_2 +\langle \xi _{F,1}X_1 + \xi _{F,2}X_2,X_2+Y_2 \rangle H_2 \\
=&\left(\left\langle \dfrac{F}{z}+\sum_{k=1}^{2}\langle F,G_k\rangle \Theta S^*H_k ,G_1 \right\rangle+ \xi _{F,1}\norm {X_1}^2+ \xi _{F,2}\langle X_2,X_1\rangle \right )H_1 \\ 
&\hspace*{.5in}+\left(\left\langle \dfrac{F}{z}+\sum_{k=1}^{2}\langle F,G_k\rangle \Theta S^*H_k ,G_2\right\rangle  + \xi _{F,2}\langle X_1,X_2 \rangle + \xi _{F,2}\norm{X_2}^2 \right)H_2 \\
=&0. ~~~~\text{(Using \ref{x} and \ref{y})}\\
\end{align*}
Thus our chosen vector $V$ fulfill all the requirements and from the construction of $V$ we  define the defect space as follows
$$ \fcl =\bigvee \{\Theta S^*H_1,\Theta S^*H_2, P_{K_\Theta}G_1,P_{K_\Theta}G_2 \} $$ with dimension atmost $4$. Consequently the kernel of $T_2$ is nearly $S^*$-invariant with defect atmost $4$.Therefore by repeating the above argument once again we have the following theorem in general context regarding the kernel of $T_n$.
\begin{thm}
	Suppose $\Phi (z)=\Theta ^*(z),$ where $\Theta \in H^\infty(\D,\mathcal{L}(\C ^m)) $ is an inner multiplier, then the following statements hold.
	\begin{enumerate}[(i)]
		\item If $G_j\in \Theta \hdcm \forall j\in\{1,2,\ldots ,n\}$, then
		 the subspace $Ker T_n$ is nearly $S^*$ invariant with defect atmost $n$ and the defect space is $$\fcl =\bigvee \{\Theta S^*H_j: j=1,2,\ldots n \}.$$
	     \item If $G_j\notin \Theta \hdcm ~for~j\in\{1,2,\ldots ,l\}\subset \{1,2,\ldots ,n\}$, then the kernel of $T_n$ is nearly $S^*$ invariant subspace of $\hdcm$ with defect atmost $n+l$ and the defect space is $$\fcl =\bigvee \{\Theta S^*H_i,P_{K_\Theta}G_j :i=1,2,\ldots ,n~and~j=1,2,\ldots l \}.$$

	\end{enumerate}
\end{thm}
 Like other sections we  now discuss the application part of Theorem \ref{a} in this context.
For the operator $T$, the equation \ref{e} is equivalent to 
\begin{align*}
& T_{\Theta ^*}F+\langle F,G\rangle H =0\\
\Leftrightarrow &~\Theta ^*F+\langle F,G\rangle H\in \overline{ H _0^2} \\
\Leftrightarrow &~F+\langle F,G\rangle \Theta H\in \Theta \overline{ H _0^2}.
\end{align*}
Observing the above equivalent criteria we  say that the kernel of the operator $T$ satisfies 
$$\mcl =Ker T\subset (H^2 \cap \Theta \overline{ H _0^2})\oplus \langle \Theta H\rangle =K_\Theta \oplus \langle \Theta H\rangle .$$
Now consider the vector $F\in \mcl =KerT$ which is of the form $F= F_\zeta +\mu \Theta H$, where $F_\zeta \in K_\Theta$ and $\mu \in \C .$ Then the above equivalent condition reduces to 
\begin{equation}\label{z}
\mu (1+\langle \Theta H,G )=-\langle F_\zeta ,G\rangle .
\end{equation}
 
Therefore from the first part of this section we conclude the defect space of the nearly $S^*$ invariant subspace $KerT$ and which is as follows:
\begin{equation*}
\fcl =
\begin{cases*}
\langle \Theta S^*H \rangle, & \text { for $G \in \Theta \hdcm ,$}\\
\bigvee \{\Theta S^*H,P_{K_\Theta}G \}, & \text{ for $G\notin \Theta \hdcm .$}
\end{cases*}
\end{equation*}
Accordingly we  analyse the kernel of $T$ in two different subsections due to the above two different representation of defect spaces.
Before we proceed we would like to mention the following: We know that the subspace $\wcl =\mcl \ominus (\mcl \cap z\hdcm )$ (if non trivial) is generated by the vectors $\{F_1,F_2,\ldots ,F_m\}$, where $F_i=P_\mcl (k_0\otimes e_i)$. If $dim \wcl =r(1\leq r\leq m)$ and $\{W_1,W_2,\ldots ,W_r \}$ is an orthonormal basis for $\wcl$, then $F_0$ is the $m\times r$ matrix whose columns are $W_1,W_2,\ldots ,W_r$. Using Gram-Schmidt orthonormalization we can form an orthonormal basis for $\wcl$ from the generator set $\{F_1,F_2,\ldots ,F_m \}$ and the resulting vectors are nothing but the linear combination of $\{F_1,F_2,\ldots ,F_m\}$. Therefore it is sufficient to prove that everything holds good for the generating set of vectors. For simplicity of calculations we do the whole analysis for single $F_i$, that means from now onwards we assume  that $dim \wcl =1$ and $F_0=[\alpha _iF_i]_{m\times 1}$ in each cases and subsections.
\subsection{\boldmath ${G \in \Theta \hdcm} $}
 $~$\\In this subsection we assume that $G\in \Theta \hdcm $. Then the corresponding defect space for the nearly $S^*$ invarint subspace $\mcl = KerT$ is $\fcl = \langle \Theta S^*H \rangle $.
 Therefore the equation \ref{z} reduces to $$\mu(1+\langle \Theta H,G\rangle)=0.$$
 Now considering the defect space as $\fcl =\langle \Theta S^* H\rangle $ we have  the following two cases.
 
 \textbf{Case 1.} If $1+\langle \Theta H,G\rangle =0$, then $\mcl =K_\Theta + \langle \Theta H \rangle . $ Moreover,
 \begin{align*}
  F_i & = P_M(k_0\otimes e_i) \\
 &=P_{K_\Theta}(k_0\otimes e_i)+P_{\langle \Theta H\rangle}(k_0\otimes e_i)\\
 &=k_0\otimes e_i -\sum_{j=1}^{m}\langle k_0\otimes e_i,\Theta e_j\rangle \Theta e_j +\dfrac{\langle k_0\otimes e_i,\Theta H\rangle \Theta H}{\norm {\Theta H}^2}\\
 &=k_0\otimes e_i -\Theta \begin{pmatrix}
 \overline{\theta_{i1}(0)}\\
 \overline{\theta_{i2}(0)}\\
 \vdots \\
 \overline{\theta_{im}(0)}\\
 \end{pmatrix} +\dfrac{\langle k_0\otimes e_i,\Theta H\rangle \Theta H}{\norm {\Theta H}^2}.
\end{align*}
 Therefore the case (i) of Theorem \ref{a} implies that the nearly invariant subspace $\mcl$ for $S^*$ with finite defect can be written in the following form.
 \begin{align*}
 &\mcl =K_\Theta +\langle \Theta H \rangle \\
 &=\Big\{F:F =F_0K_0+zk_1\dfrac{\Theta S^*H}{\norm{S^*H}} :(K_0,k_1)\in \kcl \Big \}\\
 &=\Bigg\{F:F=\alpha _i(k_0\otimes e_i -\Theta \begin{pmatrix}
 \overline{\theta_{i1}(0)}\\
 \overline{\theta_{i2}(0)}\\
 \vdots \\
 \overline{\theta_{im}(0)}\\
 \end{pmatrix} +\dfrac{\langle k_0\otimes e_i,\Theta H\rangle \Theta H}{\norm {\Theta H}^2})K_0 + k_1 \dfrac{\Theta (H-H(0))}{\norm{S^*H}} :(K_0,k_1)\in \kcl \Bigg\}\\
 &=\Bigg\{F:F=\alpha _i\Big(k_0\otimes e_i -\Theta \begin{pmatrix}
 \overline{\theta_{i1}(0)}\\
 \overline{\theta_{i2}(0)}\\
 \vdots \\
 \overline{\theta_{im}(0)}\\
 \end{pmatrix}\Big)~K_0+k_1\dfrac{\Theta H(0)}{\norm {S^*H}} \\
 &\hspace*{2in}+\Big(\alpha _i\dfrac{\langle k_0\otimes e_i,\Theta H\rangle }{\norm {\Theta H}^2}K_0 + \dfrac{k_1}{\norm{S^*H}}\Big)\Theta H :(K_0,k_1)\in \kcl \Bigg\},
 \end{align*}
 where $\kcl =\{(K_0,k_1)\in H^2_{\C^2}(\D):K_0$ and $ k_1$ satisfies the following \ref{A} \}
 \begin{equation}\label{A}
 \begin{cases}
 \alpha _i\Big(k_0\otimes e_i -\Theta \begin{pmatrix}
 \overline{\theta_{i1}(0)}\\
 \overline{\theta_{i2}(0)}\\
 \vdots \\
 \overline{\theta_{im}(0)}\\
 \end{pmatrix}\Big)~K_0+k_1\dfrac{\Theta H(0)}{\norm {S^*H}} \in K_\Theta \\
 \alpha _i\dfrac{\langle k_0\otimes e_i,\Theta H\rangle }{\norm {\Theta H}^2}K_0 + \dfrac{k_1}{\norm{S^*H}}\in \C.
 \end{cases}
 \end{equation}
 Our next aim is to show the subspace $\kcl$ of $H^2_{\C^2}(\D)$ is a $S^*\oplus S^*$ invariant subspace. First we observe that if we replace $K_0,k_1$ by $S^*K_0,S^*k_1$, then the second condition of \ref{A} is trivially true. Thus we only have to check that the first condition also true if we replace $K_0,k_1$ by $S^*K_0,S^*k_1$. Since $\mathcal{K}_\Theta$ is an $S^*$ invariant subspace of $\hdcm$, then 
 $$H_\Theta : ~ =~ S^*\left(\alpha _i\Big(k_0\otimes e_i -\Theta \begin{pmatrix}
 \overline{\theta_{i1}(0)}\\
 \overline{\theta_{i2}(0)}\\
 \vdots \\
 \overline{\theta_{im}(0)}\\
 \end{pmatrix}\Big)~K_0+k_1\dfrac{\Theta H(0)}{\norm {S^*H}}  \right)\in \mathcal{K}_\Theta.$$
 On the other hand 
 \begin{align*}
~&~\alpha _i\Big(k_0\otimes e_i -\Theta \begin{pmatrix}
\overline{\theta_{i1}(0)}\\
\overline{\theta_{i2}(0)}\\
\vdots \\
\overline{\theta_{im}(0)}\\
\end{pmatrix}\Big)~S^*K_0+ S^*k_1\dfrac{\Theta H(0)}{\norm {S^*H}}\\
&=~ H_\Theta +S^*\left(\alpha _i\Theta \begin{pmatrix}
\overline{\theta_{i1}(0)}\\
\overline{\theta_{i2}(0)}\\
\vdots \\
\overline{\theta_{im}(0)}\\
\end{pmatrix} K_0 -k_1\dfrac{\Theta H(0)}{\norm {S^*H}}\right)-\alpha _i\Theta \begin{pmatrix}
\overline{\theta_{i1}(0)}\\
\overline{\theta_{i2}(0)}\\
\vdots \\
\overline{\theta_{im}(0)}\\
\end{pmatrix}S^*K_0+ S^*k_1\dfrac{\Theta H(0)}{\norm {S^*H}}\\
&=~H_\Theta +\alpha_i \dfrac{\Theta(z)-\Theta(0)}{z}\begin{pmatrix}
\overline{\theta_{i1}(0)}\\
\overline{\theta_{i2}(0)}\\
\vdots \\
\overline{\theta_{im}(0)}\\
\end{pmatrix}K_0(0) -\dfrac{k_1(0)}{\norm{S^*H}}\dfrac{\Theta(z)-\Theta(0)}{z}H(0)\\
&=~H_\Theta +\alpha_i S^*\Theta \begin{pmatrix}
\overline{\theta_{i1}(0)}\\
\overline{\theta_{i2}(0)}\\
\vdots \\
\overline{\theta_{im}(0)}\\
\end{pmatrix}K_0(0)-\dfrac{k_1(0)}{\norm{S^*H}}S^*\Theta (H(0))\in \mathcal{K}_\Theta.
\end{align*}
 Since $H_\Theta \in \mathcal{K}_\Theta$, $S^*\Theta \begin{pmatrix}
 \overline{\theta_{i1}(0)}\\
 \overline{\theta_{i2}(0)}\\
 \vdots \\
 \overline{\theta_{im}(0)}\\
 \end{pmatrix}\in \mathcal{K}_\Theta \quad \text{and} \quad S^*\Theta (H(0))\in \mathcal{K}_\Theta$, then from the above equation we conclude that $\kcl $ is an $S^*\oplus S^*$ invariant subspace of $H^2_{\C ^2}(\D)$.\\ 
 \textbf{Case 2.} If $1+\langle \Theta H,G\rangle \neq0$, then $\mcl =\mathcal{K}_\Theta .$ Therefore\\
 $$F_i=P_\mcl (k_0\otimes e_i)=P_{\mathcal{K}_\Theta }(k_0\otimes e_i)=k_0\otimes e_i -\sum_{j=1}^{m}\langle k_0\otimes e_i,\Theta e_j\rangle \Theta e_j =k_0\otimes e_i -\Theta \begin{pmatrix}
 \overline{\theta_{i1}(0)}\\
 \overline{\theta_{i2}(0)}\\
 \vdots \\
 \overline{\theta_{im}(0)}\\
 \end{pmatrix}  .$$
 Using the case(i) of Theorem \ref{a} we have  
\begin{align*}
&\mcl =\mathcal{K}_\Theta  \\
&=\Big\{F:F =F_0K_0+zk_1\dfrac{\Theta S^*H}{\norm{S^*H}} :(K_0,k_1)\in \kcl \Big \}\\
&=\Bigg\{F:F=\alpha _i(k_0\otimes e_i -\Theta \begin{pmatrix}
\overline{\theta_{i1}(0)}\\
\overline{\theta_{i2}(0)}\\
\vdots \\
\overline{\theta_{im}(0)}\\
\end{pmatrix} )K_0 + k_1 \dfrac{\Theta (H-H(0))}{\norm{S^*H}} :(K_0,k_1)\in \kcl \Bigg\}\\
&=\Bigg\{F:F=\alpha _i\Big(k_0\otimes e_i -\Theta \begin{pmatrix}
\overline{\theta_{i1}(0)}\\
\overline{\theta_{i2}(0)}\\
\vdots \\
\overline{\theta_{im}(0)}\\
\end{pmatrix}\Big)~K_0 :(K_0,0)\in \kcl \Bigg\}.
\end{align*}
Therefore the corresponding $S^*\oplus S^* $ invariant subspace is 
$$\kcl =\{(K_0,0)\in H^2_{\C ^2}(\D) : K_0 \quad \text{satisfies the following} \quad \ref{B}\}$$
such that
	$$\alpha _i\Big(k_0\otimes e_i -\Theta \begin{pmatrix}
	\overline{\theta_{i1}(0)}\\
	\overline{\theta_{i2}(0)}\\
	\vdots \\
	\overline{\theta_{im}(0)}\\
	\end{pmatrix}\Big)~K_0 \in \mathcal{K}_\Theta ,$$
	which is equivalent to the fact that 
\begin{align}\label{B}	
	K_0\in \cap_{j=1}^m Ker T_{\overline{\theta _{ji}-\theta_{ij}(0)}}=\cap_{j=1}^mK_{\zeta_{ij}} \quad \text{where} \quad \zeta_{ij} \quad \text{is an inner factor of}\quad \theta _{ji}-\theta_{ij}(0) .
\end{align}	
Furthermore it is easy to check that $\kcl$ is an $S^*\oplus S^*$ invariant subspace of $H^2_{\C ^2}(\D)$ using the scalar version of Proposition \ref{c}.
\subsection{\boldmath ${G \notin \Theta \hdcm} $}
$~$\\In this subsection we  consider the case that $G \notin \Theta \hdcm$, then the kernel of $T$ is a nearly $S^*$ invariant subspace with defect atmost $2$ and the defect space is $\fcl =\bigvee \{\Theta S^*H,P_{K_\Theta}G \}.$
Since $G\notin \Theta \hdcm $, then we find a nonzero $G_\zeta \in K_\Theta$ and $G_\Theta\in \Theta\hdcm$ such that $$G=G_\zeta + G_\Theta .$$
Therefore the identity \ref{z} reduces to
 \begin{equation}\label{C}
 \mu (1+\langle \Theta H,G_\Theta\rangle )=-\langle F_\zeta ,G_\zeta\rangle .
 \end{equation}
To proceed further we need the following remark concerning the projection $P_\mcl(k_0\otimes e_i)$.
\begin{rmrk}
	If $\mcl =Ker T\subset \mathcal{N}:=K_\Theta \oplus \langle \Theta H\rangle $ satisfies $\mathcal{N}=\mcl \oplus \langle R \rangle $, where $R= U+\rho \Theta H$ with $U\in K_\Theta$ and $\rho \in \C$. Then
	\begin{equation}\label{D}
	P_\mcl (k_0\otimes e_i)= k_0\otimes e_i -\Theta \begin{pmatrix}
	\overline{\theta_{i1}(0)}\\
	\overline{\theta_{i2}(0)}\\
	\vdots \\
	\overline{\theta_{im}(0)}\\
	\end{pmatrix} +\dfrac{\langle k_0\otimes e_i,\Theta H\rangle \Theta H}{\norm {\Theta H}^2}-\dfrac{\langle k_0\otimes e_i,U+\rho \Theta H\rangle}{\norm{U+\rho \Theta H}^2}(U+\rho \Theta H). 
	\end{equation}
\end{rmrk}
Now we need to analyse two cases according to  whether $v_\Theta :=1+\langle \Theta H,G_\Theta\rangle $ is zero or not along with the defect space $\fcl =\bigvee \{\Theta S^*H,P_{K_\Theta}G \}.$\\
\textbf{Case 1.} If $v_\Theta =0$, then \ref{C} holds if and only if $$F_\zeta \in \langle P_{K_\Theta }G\rangle^\perp =   \langle G_\zeta \rangle^\perp .$$ 
 Thus $\mcl =\mathcal{N}\ominus \langle G_\zeta \rangle$ and therefore substituting $U=G_\zeta$ and $\rho =0$ in \ref{D} we have
 $$P_\mcl (k_0\otimes e_i)= k_0\otimes e_i -\Theta \begin{pmatrix}
 \overline{\theta_{i1}(0)}\\
 \overline{\theta_{i2}(0)}\\
 \vdots \\
 \overline{\theta_{im}(0)}\\
 \end{pmatrix} +\dfrac{\langle k_0\otimes e_i,\Theta H\rangle \Theta H}{\norm {\Theta H}^2}-\dfrac{\langle k_0\otimes e_i,G_\zeta\rangle}{\norm{G_\zeta}^2}G_\zeta. $$
 Therefore $F_0=[\alpha_iP_\mcl(k_0\otimes e_i)]_{m\times 1}$ and hence by using the case (i) of Theorem \ref{a}, we have the follwing represention of $\mcl$: 
 \begin{align*}
 \mcl &=\Bigg\{F:F = \alpha_i\Bigg( k_0\otimes e_i -\Theta \begin{pmatrix}
 \overline{\theta_{i1}(0)}\\
 \overline{\theta_{i2}(0)}\\
 \vdots \\
 \overline{\theta_{im}(0)}\\
 \end{pmatrix} +\dfrac{\langle k_0\otimes e_i,\Theta H\rangle \Theta H}{\norm {\Theta H}^2}-\dfrac{\langle k_0\otimes e_i,G_\zeta\rangle}{\norm{G_\zeta}^2}G_\zeta\Bigg)  K_0 \\ 
 &\hspace*{2in}+ zk_1\dfrac{\Theta S^*H}{\norm {S^*H}}+zk_2\dfrac{G_\zeta }{\norm {G_\zeta }} :(K_0,k_1,k_2)\in \kcl\Bigg\}\\
 & = \Bigg\{F:F =\alpha_i \Bigg( k_0\otimes e_i -\Theta \begin{pmatrix}
 \overline{\theta_{i1}(0)}\\
 \overline{\theta_{i2}(0)}\\
 \vdots \\
 \overline{\theta_{im}(0)}\\
 \end{pmatrix} \Bigg)  K_0-\dfrac{k_1}{\norm {S^*H}}\Theta H(0)+ \Bigg(\dfrac{k_1}{\norm {S^*H}}+\dfrac{\langle k_0\otimes e_i,\Theta H\rangle }{\norm {\Theta H}^2}\Bigg)\Theta H \\ 
 &\hspace*{2in}-\Bigg(\dfrac{\langle k_0\otimes e_i,G_\zeta\rangle \alpha_i K_0}{\norm{G_\zeta}^2}-\dfrac{zk_2 }{\norm {G_\zeta }}\Bigg)G_\zeta :(K_0,k_1,k_2)\in \kcl\Bigg\}\\
 \end{align*}
and the corresponding $S^*\oplus S^*\oplus S^*$ invariant subspace is $$ \kcl =\Bigg\{(K_0,k_1,k_2)\in H^2_{\C^3}(\D):K_0,k_1,k_2 \text{~satisfies ~the~ following ~\ref{E}}\Bigg\},$$ where
 \begin{equation}\label{E}
 \begin{cases}
 \alpha_i \Bigg( k_0\otimes e_i -\Theta \begin{pmatrix}
 \overline{\theta_{i1}(0)}\\
 \overline{\theta_{i2}(0)}\\
 \vdots \\
 \overline{\theta_{im}(0)}\\
 \end{pmatrix} \Bigg)  K_0-\dfrac{k_1}{\norm {S^*H}}\Theta H(0)\in \mathcal{K}_\Theta \\
 \dfrac{k_1}{\norm {S^*H}}+\dfrac{\langle k_0\otimes e_i,\Theta H\rangle }{\norm {\Theta H}^2}\in \C \\
 \langle K_0,z^n\overline{ \alpha_i}G_i\rangle - \Bigg\langle\Bigg(\dfrac{\langle k_0\otimes e_i,G_\zeta\rangle \alpha_iS^{*n} K_0}{\norm{G_\zeta}^2}-\dfrac{zS^{*n}k_2 }{\norm {G_\zeta }}\Bigg)G_\zeta,G_\zeta\Bigg\rangle =0, ~for~n\in \mathbb{N}\cup \{0\}
 \end{cases}
 \end{equation}
and $G_\zeta =(G_1,G_2,\ldots ,G_m)$.
In a similar fashion like \ref{A} we  prove that  the first two conditions of \ref{E} also hold for $S^*K_0$, $S^*k_1$ . Moreover, the last condition also holds for $S^*K_0$ and $S^*k_2$ trivially. Thus $\kcl $ is an $S^*\oplus S^*\oplus S^*$ invariant subspace of $H^2_{\C ^3}(\D)$.\\
\textbf{Case 2.} Next we consider $v_\Theta \neq 0$, then the equation \ref{E} gives $\mu =-v_\Theta ^{-1}\langle F_\zeta ,G_\zeta\rangle$ and using \ref{B} and \ref{E} we have, $$ \mcl =Ker T=\{F: F=K -v_\Theta ^{-1}\langle K, G_\zeta \rangle \Theta H ,K\in K_\Theta \}.$$
Therefore by simple calculations we conclude that $\mathcal{N}=\mcl \oplus \langle G_\zeta +\dfrac{ \overline{v}_\Theta}{\norm H ^2} \Theta H \rangle$. Thus by letting $U=G_\zeta $ and $\rho =\dfrac{ \overline{v}_\Theta}{\norm H ^2}$ in (5.20), we get
\begin{align*}
P_\mcl (k_0\otimes e_i)&= k_0\otimes e_i -\Theta \begin{pmatrix}
\overline{\theta_{i1}(0)}\\
\overline{\theta_{i2}(0)}\\
\vdots \\
\overline{\theta_{im}(0)}\\
\end{pmatrix} +\dfrac{\langle k_0\otimes e_i,\Theta H\rangle \Theta H}{\norm {\Theta H}^2}-\dfrac{\langle k_0\otimes e_i,G_\zeta+\dfrac{ \overline{v}_\Theta}{\norm H ^2} \Theta H\rangle}{\norm{G_\zeta+\dfrac{ \overline{v}_\Theta}{\norm H ^2} \Theta H}^2}(G_\zeta+\dfrac{ \overline{v}_\Theta}{\norm H ^2} \Theta H) \\
&=k_0\otimes e_i -\Theta \begin{pmatrix}
\overline{\theta_{i1}(0)}\\
\overline{\theta_{i2}(0)}\\
\vdots \\
\overline{\theta_{im}(0)}\\
\end{pmatrix} +\dfrac{\langle k_0\otimes e_i,\Theta H\rangle \Theta H}{\norm { H}^2}-\omega _\Theta(G_\zeta+\dfrac{ \overline{v}_\Theta}{\norm H ^2} \Theta H),
\end{align*}
where $$\omega _\Theta =\dfrac{\langle k_0\otimes e_i,G_\zeta+\dfrac{ \overline{v}_\Theta}{\norm H ^2} \Theta H\rangle}{\norm{G_\zeta+\dfrac{ \overline{v}_\Theta}{\norm H ^2} \Theta H}^2}. $$
Therefore from the Case (i) of Theorem \ref{a} we have the following representation of $\mcl$:

\begin{align*}
\mcl &= \Bigg(K_\Theta \oplus \langle \Theta H\rangle \Bigg)\ominus \Bigg\langle G_\zeta +\dfrac{ \overline{v}_\Theta}{\norm H ^2} \Theta H \Bigg\rangle \\
&=\Bigg\{F:F=\alpha_i\Bigg( k_0\otimes e_i -\Theta \begin{pmatrix}
\overline{\theta_{i1}(0)}\\
\overline{\theta_{i2}(0)}\\
\vdots \\
\overline{\theta_{im}(0)}\\
\end{pmatrix} +\dfrac{\langle k_0\otimes e_i,\Theta H\rangle \Theta H}{\norm { H}^2}-\omega _\Theta(G_\zeta+\dfrac{ \overline{v}_\Theta}{\norm H ^2} \Theta H)\Bigg)K_0 \\ 
&\hspace*{3in}+k_1\dfrac{\Theta (H-H(0))}{\norm {S^*H}}+k_2\dfrac{zG_\zeta}{\norm{G_\zeta}}:(K_0,k_1,k_2)\in \kcl \Bigg\}\\
&=\Bigg\{F:F=\alpha_i(k_0\otimes e_i)K_0-\Bigg( \alpha_iK_0\Theta \begin{pmatrix}
\overline{\theta_{i1}(0)}\\
\overline{\theta_{i2}(0)}\\
\vdots \\
\overline{\theta_{im}(0)}\\
\end{pmatrix}+k_1\dfrac{\Theta H(0)}{\norm {S^*H}} \Bigg)  \\
&\hspace*{2in}+\Bigg(\alpha_i \dfrac{\langle k_0\otimes e_i,\Theta H\rangle }{\norm { H}^2}K_0+\dfrac{k_1}{\norm {S^*H}}-\dfrac{zk_2}{\norm{G_\zeta}}\dfrac{ \overline{v}_\Theta}{\norm H ^2}\Bigg)\Theta H \\
&\hspace*{2in} +\Bigg( -\alpha_iK_0\omega_\Theta +\dfrac{zk_2}{\norm{G_\zeta}} \Bigg)(G_\zeta+\dfrac{ \overline{v}_\Theta}{\norm H ^2} \Theta H) :(K_0,k_1,k_2)\in \kcl\Bigg\}
\end{align*}
and the corresponding $S^*\oplus S^*\oplus S^*$ invariant subspace is $$ \kcl =\{(K_0,k_1k_2):K_0,k_1,k_2 \quad \text{satisfies the following } \quad \ref{F} \},$$ where
\begin{equation}\label{F}
\begin{cases}
\alpha_i(k_0\otimes e_i)K_0-\Bigg( \alpha_iK_0\Theta \begin{pmatrix}
\overline{\theta_{i1}(0)}\\
\overline{\theta_{i2}(0)}\\
\vdots \\
\overline{\theta_{im}(0)}\\
\end{pmatrix}+k_1\dfrac{\Theta H(0)}{\norm {S^*H}} \Bigg) \in \mathcal{K}_\Theta \\
\alpha_i \dfrac{\langle k_0\otimes e_i,\Theta H\rangle }{\norm { H}^2}K_0+\dfrac{k_1}{\norm {S^*H}}-\dfrac{zk_2}{\norm{G_\zeta}}\dfrac{ \overline{v}_\Theta}{\norm H ^2}\in \C \\
\langle K_0,z^n\overline{ \alpha_i}G_i\rangle+ const +\Bigg\langle\Bigg( -\alpha_iS^{*n}K_0\omega_\Theta +\dfrac{zS^{*n}k_2}{\norm{G_\zeta}} \Bigg)(G_\zeta+\dfrac{ \overline{v}_\Theta}{\norm H ^2} \Theta H), G_\zeta+\dfrac{ \overline{v}_\Theta}{\norm H ^2} \Theta H \Bigg\rangle =0,\\ \hspace*{5in}~for~n\in \mathbb{N}\cup \{0\},
\end{cases}
\end{equation}
and  $G_\zeta =(G_1,G_2,\ldots ,G_m)$.
By repeating the similar explanations as in (\ref{A}), the conditions of (\ref{F}) also hold for $S^*K_0,S^*k_1,S^*k_2$ and hence we conclude that, $\kcl $ is an $S^*\oplus S^*\oplus S^*$ invariant subspace of $H^2_{\C ^3}(\D)$.
Now we give an example to understand the case $\Phi =\Theta ^*$,  where $\Theta $ an inner multiplier.
\begin{xmpl}
	Suppose \begin{equation}
\Theta  (z)=
	\begin{bmatrix}  
	z^s&0&\cdots&0 \\
	0&z^s&\cdots&0 \\
	\vdots&\vdots&\ddots&\vdots\\
	0&0&\cdots&z^s \\
	\end{bmatrix}_{m\times m} \in H^\infty(\D,\mathcal{L}(\C ^m)) , ~where~s\geq 1,
	\end{equation}
 $G=G_\zeta +G_\Theta$ with $G_\zeta =(z^{s-1},1,\ldots ,1) ,G_\Theta \in \Theta\hdcm=z^s\hdcc \oplus \ldots \oplus z^s\hdcc$ and $H= \dfrac{1}{\sqrt{2}}(1-z,0,\ldots ,0) $. It then follows that, $\mcl$ is a nearly $S^*$ invariant subspace with defect space $$\fcl =\bigvee \{z^sS^*H,G_\zeta\}.$$
 \textbf{Case 1.} If $v_\Theta =1+\langle \Theta H,G_\Theta \rangle =0 $, then it follows that 
 \begin{align*}
 \mcl &= \bigvee \{1\otimes e_1,\{z\otimes e_i\}_{i=1}^m,\{z^2\otimes e_i\}_{i=1}^m,\ldots ,\{z^{s-2}\otimes e_i\}_{i=1}^m,\{z^{s-1}\otimes e_j \}_{j=2}^m \}\oplus \langle z^sH\rangle \\
 &=\Bigg\{F:F =   K_0\otimes e_1-k_1\dfrac{z^sH(0)}{\norm {S^*H}}+ k_1\dfrac{z^s H}{\norm {S^*H}}+zk_2\dfrac{G_\zeta }{\norm {G_\zeta }} :(K_0,k_1,k_2)\in \kcl\Bigg\}
 \end{align*}
with an $S^*\oplus S^*\oplus S^*$ invariant subspace $$ \kcl =\{(K_0,k_1,k_2):K_0,k_1 \quad \text{satiesfies the following} \quad (\ref{jk}) \quad \text{and}\quad  k_2\in \hdcc  \} $$
such that 
\begin{equation}\label{jk}
\begin{cases}
K_0\otimes e_1 -k_1\dfrac{z^sH(0)}{\norm {S^*H}} \in \mathcal{K}_\Theta \\
\dfrac{k_1}{\norm {S^*H}}\in \C \\
\langle K_0, z^{n+s-1}\rangle  =0 ~for~n\in \mathbb{N}\cup \{0\}.
\end{cases}
\end{equation}
\textbf{Case 2.} If $v_\Theta =1+\langle \Theta H,G_\Theta \rangle \neq 0 $, then it follows that 
\begin{align*}
\mcl &= \bigvee \Bigg\{1\otimes e_1,\{z\otimes e_i\}_{i=1}^m,\{z^2\otimes e_i\}_{i=1}^m,\ldots ,\{z^{s-2}\otimes e_i\}_{i=1}^m,\{z^{s-1}\otimes e_j \}_{j=1}^m \Bigg\}\oplus \langle z^sH \rangle\\
&\hspace*{5in} \ominus \langle G_\zeta +\dfrac{ \overline{v}_\Theta}{\norm H ^2} z^s H \rangle \\
&=\Bigg\{F:F=K_0\otimes e_1-k_1\dfrac{z^s H(0)}{\norm {S^*H}}+\Bigg(\dfrac{k_1}{\norm {S^*H}}-\dfrac{zk_2}{\norm{G_\zeta}}\overline{v}_\Theta\Bigg)z^s H \\
&\hspace*{3in} +\dfrac{zk_2}{\norm{G_\zeta}}(G_\zeta+\overline{v}_\Theta z^s H) :(K_0,k_1,k_2)\in \kcl\Bigg\}
\end{align*}
with an $S^*\oplus S^*\oplus S^*$ invariant subspace $$ \kcl =\{(K_0,k_1,k_2):K_0,k_1,k_2 \quad \text{satiesfies the  following}\quad (\ref{kj})   \} $$
such that 
\begin{equation}\label{kj}
\begin{cases}
K_0\otimes e_1 -k_1\dfrac{z^sH(0)}{\norm {S^*H}} \in \mathcal{K}_\Theta \\
\dfrac{k_1}{\norm {S^*H}}-\dfrac{zk_2}{\norm{G_\zeta}}\overline{v}_\Theta\in \C \\
\Bigg\langle K_0, z^{n}\Bigg(z^{s-1}+\dfrac{\overline{v}_\Theta z^s}{\sqrt{2}}(1-z)\Bigg)\Bigg\rangle +\langle k_1,\dfrac{\overline{v}_\Theta}{\sqrt{2}}z^n \rangle  +\langle k_2,\dfrac{z^n\overline{v}_\Theta }{\sqrt{2}}(1-z) \rangle=0 ~for~n\in \mathbb{N}\cup \{0\}.
\end{cases}
\end{equation}
\end{xmpl}

\section*{Acknowledgements}
\textit{ We would like to thank Prof. Joydeb Sarkar and  Dr. Bata Krishna Das for introducing this area to us.}

\end{document}